\documentclass[reqno]{amsart}
\usepackage{amsmath}
\usepackage{amscd}
\usepackage{graphics}
\usepackage{latexsym}
\usepackage{color}
\usepackage{verbatim}
\usepackage{extarrows}
\usepackage{latexsym}
\usepackage{enumerate}
\usepackage{arydshln}
\usepackage{soul}
\usepackage{tikz,pgfplots}
\usepackage{tikz-cd}
\usepackage{ulem} 
\definecolor{darkgreen}{HTML}{3CB50F}
\usetikzlibrary{arrows,calc,backgrounds}
\usepackage[margin=1in]{geometry}
\usepackage[all]{xy}
\theoremstyle{plain}
\newtheorem{theorem}{Theorem}[section]

\newtheorem{cor}[theorem]{Corollary}
\newtheorem{lem}[theorem]{Lemma}
\newtheorem{prop}[theorem]{Proposition}

\theoremstyle{definition}
\newtheorem{definition}[theorem]{Definition}

\newtheorem{conj}[theorem]{Conjecture}

\newtheorem{rmk}[theorem]{Remark}

\theoremstyle{theorem}
\newtheorem*{thm1}{Theorem A}
\newtheorem*{thm2}{Theorem B}
\newtheorem*{coro}{Corollary}

\theoremstyle{remark}

\newcommand{\ZZ}{\mathbb{Z}}

\newcommand{\PP}{\mathbb{P}}
\newcommand{\EFF}{\mathrm{Eff}}

\newcommand{\Mov}{\mathrm{Mov}}

\newcommand{\MOV}{\mathrm{Mov}}

\newcommand{\OO}{\mathcal O}
\newcommand{\OOT}{\mathcal{O}_{\mathbb{P}^2}}

\newcommand{\Hom}{\operatorname{Hom}}

\def\Hom{{\rm Hom}\,}

\def\PP{{\mathbb P}}

\pgfplotsset{compat=1.15}
\begin{document}

\title{Minimal free resolutions of sheaves on the projective plane and the stable base locus decomposition of their moduli spaces} 
\author{Manuel Leal}
\author{C\'esar Lozano Huerta}
\author{Tim Ryan}


\address{Universidad Nacional Aut\'onoma de M\'exico\\
Instituto de Matem\'aticas, Unidad Oaxaca \\
Oaxaca, Mex.}
\email{maz.leal.camacho@gmail.com}
\email{lozano@im.unam.mx}

\address{University of Michigan\\
Department of mathematics \\
Ann Arbor, USA.}
\email{rtimothy@umich.edu} 


\date{October 2021.}

\maketitle
\begin{abstract}
The purpose of this paper is to incorporate minimal free resolutions into the study of the birational geometry of the moduli space of coherent sheaves on the plane with character $\xi$, denoted by $M(\xi)$. We show that it is possible to recover the relevant Bridgeland destabilizing object from the minimal free resolution in order to compute the effective cone $\EFF(M(\xi))$ and conjecture a full relationship between Bridgeland destabilizing objects and minimal free resolutions. Moreover, we also prove that minimal free resolutions, paired with interpolation for vector bundles, yield the movable cone of the Hilbert scheme of $n$ points on the plane $\PP^{2[n]}$, for certain values of $n$. We propose a program that computes the full stable base locus decomposition of $\PP^{2[n]}$ based on free resolutions and interpolation. We exhibit that this programs yields correct answers for small values of $n$. 

 \end{abstract}

\bigskip
\section{Introduction}

\medskip\noindent
The purpose of this paper is to incorporate minimal free resolutions into the study of the birational geometry of the moduli space of coherent sheaves on the plane with character $\xi$, denoted by $M(\xi)$. There are at least two approaches to study the stable base locus decomposition of $M(\xi)$: Bridgeland stability and interpolation, and this paper shows how the minimal free resolution interacts with both of them.

\medskip\noindent
The minimal free resolution of a coherent sheaf on the plane is a classical tool in algebraic geometry and commutative algebra and its main two ingredients are a Betti diagram and a matrix with polynomial entries. 
It is intuitively clear that there should be a relationship between these resolutions and the stable base locus decomposition (SBLD) of the effective cone $\EFF(M(\xi))$. 
However, thus far, the exact relationship has remained unclear and recent work has used Bridgeland stability conditions, via destabilizing objects, to compute the SBLD of $\EFF(M(\xi))$ without any reference to minimal free resolutions.
We then ask:

\medskip\noindent
\textbf{Question 1:} \textit{Is it possible to recover the relevant Bridgeland destabilizing objects from the minimal free resolutions of the elements in $M(\xi)$ to compute the stable base locus decomposition of $\EFF(M(\xi))$?}

\medskip\noindent
Our first result, Theorem A, starts the study of this question and asserts that it is indeed possible to recover the relevant Bridgeland destabilizing object from the minimal free resolution in order to compute the effective cone $\EFF(M(\xi))$. Further, if $\PP^{2[n]}$ denotes the Hilbert scheme of $n$ points on the plane, then Theorem B exhibits that it is also possible to recover the Bridgeland destabilizing objects from minimal free resolutions to get the movable cone $\MOV(\PP^{2[n]})$, for certain (infinite) values of $n$. Moreover, if we examine $\PP^{2[n]}$, with $n=2,3,4,5,6,12$, then the answer to Question 1 is affirmative as we see in Section \ref{section5}. Section \ref{section5} also provides a precise conjecture to answer Question 1.

\medskip\noindent
From a more classical viewpoint, one may wish to recover the SBLD of $\EFF(M(\xi))$ directly from the minimal free resolutions without needing to appeal to Bridgeland stability.  This can be achieved via the interpolation problem for vector bundles stated later in this introduction.

\medskip\noindent
In \cite{CHW}, the bundles that solve this interpolation problem, called interpolating bundles, are used to describe the edges of the effective cone $\EFF(M(\xi))$. For this description, a resolution of a general sheaf $U\in M(\xi)$ in terms of higher rank vector bundles is of key importance. Naturally, one can ask if one can replace the use of vector bundle resolutions by minimal free resolutions:

\medskip\noindent
\textbf{Question 2:} \textit{Is it possible to use interpolating bundles and the minimal free resolutions of the elements in $M(\xi)$ to compute the stable base locus decomposition of $\EFF(M(\xi))$?}

\medskip\noindent
Theorem A and the corollaries below partially answer Question 2. Indeed, they describe the stable base locus of the (primary) extremal chamber of $\EFF(M(\xi))$ when such base locus is divisorial. Furthermore, Theorem B asserts that also the movable cone of $\PP^{2[n]}$, for certain values of $n$, can be described from the minimal free resolution and interpolation; which contributes to answering Question 2 in the affirmative. In Section \ref{section5}, we exhibit that the answer to Question 2 is affirmative for $\PP^{2[n]}$ when $n=2,3,4,5,6,12$. In order to carry out these computations of the SBLD systematically, we propose an \textit{interpolation program}, which we describe later in this introduction and which could be applied more generally.

\begin{center} 
   \includegraphics[scale=.15]{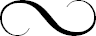} 
\end{center}

\medskip\noindent
Let us now describe our results in more detail. In doing so, we revisit work by Coskun, Huizenga and Woolf \cite{CHW} and reinterpret some of their results in terms of minimal free resolutions.  We now recall the description of $\EFF(M(\xi))$, by Coskun et al., in terms of the interpolation problem for vector bundles.

\medskip\noindent
\textbf{Interpolation problem for vector bundles:} Given a sheaf $U\in M(\xi)$, determine the infimum of slopes $\mu^+\in \mathbb{Q}$ such that
\begin{enumerate}
\item[(a)] $\mu(\xi)+\mu^+\ge 0$ and
\item[(b)] $U$ is cohomologically orthogonal to a vector bundle $V$ of slope $\mu^+$.
\end{enumerate}

\medskip\noindent
If $U\in M(\xi)$ is a general sheaf, then by solving the interpolation problem for $U$, one finds two bundles: the interpolating bundle $V$, cohomologically orthogonal to $U$; and in addition an exceptional bundle $E_{\alpha\cdot \beta}$, called the controlling exceptional bundle, which is determined from the character $\xi$. 
If $U\otimes V$ has no cohomology, then the Brill-Noether divisor $$D_V=\{F\in M(\xi)\ | \ h^1(F\otimes V)\ne 0\} $$ is an effective divisor which spans an extremal class in $\EFF(M(\xi))$. 
In order to show that $U\otimes V$ has in fact no cohomology, Coskun et al. use the so-called \textit{Gaeta triangle} of $U$, which is an element of the derived category $D^b(\PP^2)$ and, in their notation, can be written in one of the following forms
\begin{equation}\label{gaetatriangle}
    E^m_{-(\alpha\cdot\beta)}\rightarrow U\rightarrow W[1] \rightarrow \cdot,  \qquad  W \rightarrow U \rightarrow E_{-(\alpha\cdot\beta)-3}^{m}[1]\rightarrow \cdot, \qquad\text{ or } \qquad E^l_{-\beta} \rightarrow U\rightarrow E^k_{-\alpha-3}[1] \rightarrow \cdot
\end{equation}

\medskip\noindent
It follows that $U\otimes V$ has no cohomology if the tensor products of $V$ with each of the factors of the Gaeta triangle (\textit{e.g.}, $E_{-(\alpha\cdot\beta)} \otimes V$ and $W\otimes V$) has no cohomology, and this is what they proved. Therefore, $D_V$ is an effective divisor which spans an extremal class in $\EFF(M(\xi))$.

\medskip\noindent
Our first result is a dictionary between a Gaeta triangle of $U$ and the classical Gaeta minimal free resolution of $U$ (see Section \ref{preliminaries} for definitions). We show that a Gaeta triangle is equivalent to the Gaeta minimal free resolution with a map whose properties depend on the character of $U$. Let us phrase this result as follows and refer the reader to Theorem \ref{MAINdetailed} for the details.

\medskip\noindent
\begin{thm1} \label{thm: main}
A stable sheaf $U\in M(\xi)$ fits into a Gaeta triangle above (\ref{gaetatriangle}) if and only if its minimal free resolution has the Gaeta Betti diagram and, in addition, the map in it contains the map defining the minimal free resolution of the following:
\begin{enumerate}
    \item[(a)] the bundle $E_{-(\alpha \cdot \beta)}^m$ if $U$ fits into the triangle in \eqref{gaetatriangle} on the left,
    \item[(b)] the sheaf $W$ if $U$ fits into the triangle in \eqref{gaetatriangle} in the middle,
    \item[(c)] the bundle $E_{-\beta}^\ell$ if $U$ fits into the triangle in \eqref{gaetatriangle} on the right. 
\end{enumerate}
\end{thm1}

\medskip\noindent
The strategy to prove Theorem A is the following: we interpret a Gaeta triangle (\textit{e.g.}, $E_{-(\alpha\cdot \beta)}^{m}\rightarrow U\rightarrow W[1] \rightarrow\cdot $) as a map of complexes in $D^b(\PP^2)$ whose terms are sums of line bundles (namely, a map between the minimal free resolutions of $W$ and $E_{-(\alpha\cdot \beta)}$). The mapping cone of this  map yields a complex which involves only line bundles and we argue that this is, in fact, the minimal free resolution of $U$. 
Since a Gaeta triangle has one of three possible forms, the proof of this result has to analyze three possibilities, which we do in \S \ref{Positivecase}, \S \ref{NegativeCase}, and \S \ref{sec: exceptional case}.

\medskip\noindent
Theorem A allows us to study base loci of divisor classes. For example, if a general sheaf $U\in M(\xi)$ fits into the triangle in (\ref{gaetatriangle}) on the right $i.e.$, we are in the situation of Theorem A part (c), then we can write the minimal free resolution of a generic sheaf $F$ in the stable base locus of the (primary) extremal divisor class $J$ of the effective cone $\EFF(M(\xi))$. This is the content of the following corollary; we refer the reader to Corollary \ref{cor: EffRigid} for the details.

\begin{coro}
Suppose $U\in M(\xi)$ fits into the triangle in (\ref{gaetatriangle}) on the right. Then, the closure in $M(\xi)$ of the locus of sheaves which fail to contain $E_{-\beta}$ in its minimal free resolution forms an irreducible and reduced divisor which is the stable base locus of the primary extremal chamber of $\EFF(M(\xi))$.
\end{coro}

\medskip\noindent
When the Gaeta triangle referred to in the previous corollary coincides with the Gaeta minimal free resolution, then the previous result reads as follows. This is the content of Corollary \ref{cor: DGaeta}.

\begin{coro}
If the Gaeta resolution of $U\in M(\xi)$ is pure, then the closure in $M(\xi)$ of the locus of sheaves whose Betti diagram is not that of Gaeta forms an irreducible and reduced divisor. 
\end{coro}

\medskip\noindent
In the case $M(\xi)=\PP^{2[n]}$, we have a converse of the previous corollary; see Corollary \ref{cor: DGaeta2}. 

\begin{coro}
Let $n$ be such that the closure in $\PP^{2[n]}$ of the locus of sheaves whose Betti diagram is not that of Gaeta, called $D_{\tiny{Gaeta}}$, has codimension 1. Then, the Gaeta minimal free resolution is pure.
\end{coro}

\medskip\noindent
 The previous corollaries contribute to answering Question 2 in the affirmative as far as the edge of the effective cone is concern. In this same setting, we now investigate whether the inner walls of the SBLD of $\EFF(M(\xi))$ are also determined by information contained in the minimal free resolution.
A priory, there is no reason why this should be the case as there are many families that naturally arise in the geometry of sheaves that are not controlled by conditions on the minimal free resolution; see \cite{RYANLOZANOHUERTA} for examples of this type. 
However, we are able to prove the following.
 

\medskip\noindent
We call a positive integer \textit{triangular} if it is of the form $n=\tfrac{r(r+1)}{2}$; likewise, a positive integer of the form $n=2s(s+1)$ will be called a \textit{tangential} number. These are precisely the cases for which the Gaeta triangle in (\ref{gaetatriangle}) on the right coincides with the classical Gaeta resolution.

\medskip\noindent
\begin{thm2} 
Let $J$ be the primary extremal divisor of $M\left(1,0,-n\right)=\PP^{2[n]}$, with $n$ a triangular or a tangential number. 
Then the interpolation problem is solvable for a general $\Gamma \in J$ by a stable vector bundle $M$ whose presentation can be seen in Theorems \ref{TRI} and \ref{MOV}. 
Consequently, the Brill-Noether divisor $D_M$ spans a movable class, which in addition is extremal in the movable cone $\MOV(\PP^{2[n]})$.
\end{thm2}

\medskip\noindent
The previous statement is the content of Theorem \ref{TRI} and Theorem \ref{MOV}, where we also write down the class of the extremal divisor of $\MOV({\PP^{2[n]}})$. We prove these theorems by using Corollary \ref{cor: DGaeta} (the second corollary above) and by solving the interpolation problem for a general element in the the extremal divisor $J\in \PP^{2[n]}$. Indeed, we find the minimal slope bundle $M$ which is cohomologically orthogonal to a general $\Gamma\in J $, by looking at the minimal free resolution of the ideal sheaf $\mathcal{I}_{\Gamma}$. Consequently, the induced Brill-Noether divisor $D_{M}$ spans an effective class which does not have $J$ in its base locus; hence $D_M\in \MOV(\PP^{2[n]})$. We prove that $D_M$ is extremal in the movable cone by exhibiting a curve dual to $D_M$ which sweeps out an open set of $J$.  We observe that this procedure fits into the following elementary program.

\medskip\noindent 
\textbf{Interpolation Program:} In order to find the inner walls of the SBLD of $\EFF(\PP^{2[n]})$,  we may proceed as follows. Let $J\subset \PP^{2[n]}$ be the primary extremal divisor of $\EFF(\PP^{2[n]})$. Then,
\begin{enumerate}
\item Solve the interpolation problem for a generic point in the stable base locus of $J$ using its minimal free resolution. The outcome is a bundle $M$.
\item The Brill-Noether divisor induced by $M$, denoted $D_M$, spans a potential wall in the SBLD. 
\item The divisor $D_M$ has a base locus of its own, so we may repeat the item (1) replacing $J$ with $D_M$.
\end{enumerate}

\medskip\noindent
We run this program in Section 5 successfully and exhibit that both of our questions above are answered in the affirmative in the cases of the Hilbert scheme $\PP^{2[n]}$ with small values of $n$; in particular this program recovers some of the examples of \cite{ABCH}.

\medskip\noindent
A comment is in order about how feasible this program is. First, the stable base locus of a divisor is in general reducible. Then, the item (1) has to be done component by component. Second, the item (2) yields a class $D_M$ which does not have the points considered in item (1) in its base locus. To show that $D_M$ spans an actual wall in the SBLD, it suffices to exhibit a dual curve that sweeps out the component considered in item (1). Finding this curve may be challenging (see Remark \ref{eg12}). 
Third, and more pressing point, we do not know if after running this program we have exhausted all of the walls of the SBLD as we do not know if all such walls are determined by imposing conditions on the minimal free resolution. This is the content of our Question 2.

\medskip\noindent
\textbf{Bridgeland stability viewpoint:} The inner walls of the SBLD of $\PP^{2[n]}$ are known using Bridgeland stability conditions \cite{LZ,BM}. 
Answering Question 1 above then reduces to the existence of a dictionary between the destabilizing objects (triangles) and (some of the) minimal free resolutions. 
The Gaeta triangle yields one such destabilizing object (triangle) $F \rightarrow U$ and Theorem A establishes the relation between this object and the minimal free resolution of $U$. Indeed, in this language, Theorem A claims that if $U\in M(\xi)$ is general, then $U$ is $\zeta$-admissible (see Definition \ref{generalgaeta}), where $\zeta$ is the destabilizing character of $U$. 

\medskip\noindent
Theorem B also provides evidence that a destabilizing object that gives rise to the movable cone is contained in the minimal free resolution of sheaves. Section \ref{section5} works out explicit examples where this is again the case for all destabilizing objects that occur. 
In fact, Conjecture \ref{conj} aims at making the dictionary between destabilizing objects and minimal free resolutions precise and claims that one can recover the destabilizing objects from the minimal free resolutions in general.

\section*{organization of the paper} \noindent
In Section \ref{preliminaries} we include preliminaries on sheaves over the plane and their minimal free resolutions. We also include the notation of the objects in birational geometry we aim to study using free resolutions.
In Section \ref{sec: eff}, we prove Theorem A, which in this part of the paper is labeled as Theorem \ref{MAINdetailed}. 
Since a Gaeta triangle has three possible forms, we have to analyze three cases in order to prove Theorem \ref{MAINdetailed}; we do so in \S \ref{Positivecase}, \S \ref{NegativeCase}, and \S \ref{sec: exceptional case}.
In Section \ref{sec: mov}, we prove Theorem B. Notice that this result involves two cases: triangular and tangential numbers. Consequently, we prove Theorem B in two steps: Theorem \ref{TRI} and Theorem \ref{MOV}. We want to point out that the core ideas in proving these two theorems are distinct in each case. In Section \ref{section5}, we show that we can answer both Questions 1 and 2 for $\mathbb{P}^{2[n]}$ with $n=2,3,4,5,6,12$.


\bigskip
\section{Preliminaries}\label{preliminaries}
\medskip\noindent
This section recalls the notions about sheaves on $\PP^2$ that will be used throughout the paper.
We work over $\mathbb{C}$ throughout and all sheaves are coherent.

\subsection{The logarithmic Chern character}

\noindent
Let $U$ and $V$ be  coherent sheaves on $\PP^2$ with $U$ having Chern character $(r,\mathrm{ch}_1,\mathrm{ch}_2)$ and $r>0$.
The \textit{slope} and \textit{discriminant} of $U$  are defined to be $\mu(U)=\frac{\mathrm{ch}_1}{r} \;\text{ and }\; \Delta(U) =\frac{1}{2}\mu(U)^2-\frac{\mathrm{ch}_2}{r}$.
The slope and discriminant are logarithmic Chern classes in the sense that \[\mu(U\otimes V) = \mu(U) +\mu(V) \qquad \text{ and }\qquad \; \Delta(U\otimes V) = \Delta(U) +\Delta(V).\] Similarly, for the dual bundle $U^*$ we have \[\mathrm{rk}(U^*) =\mathrm{rk}(U), \qquad \mu(U^*) = -\mu(U),\qquad \text{ and } \qquad \Delta(E^*) = \Delta(E).\] 
Then, we will abuse notation and refer to $(r,\mu,\Delta)$ as the log Chern character of a sheaf $U$ of positive rank. 

\medskip\noindent
A coherent, torsion-free sheaf is called \textit{slope (semi-)stable} if all proper nontrivial subsheaves have smaller (or equal) slope. 

\subsection{Euler Characteristic of sheaves}
\medskip\noindent
By the Riemann-Roch formula, the Euler characteristic of $U$ is 
\[ \chi(U) = \sum_{i=0}^2(-1)^i h^i\left(\mathbb{P}^2,U\right) = \mathrm{rk}(U)\left(\frac{1}{2}\left(\mu(U)+1\right)\left(\mu(U)+2\right)-\Delta(U)\right).\]

\medskip\noindent
Similarly, the \textit{relative Euler characteristic} of two sheaves $U$ and $V$ is 
\begin{align}\label{CHI} 
\begin{split}
\chi(U,V) &= \sum_{i=0}^2(-1)^i \mathrm{ext}^i\left(U,V\right) \\         &=\mathrm{rk}(U)\mathrm{rk}(V)\left(\frac{1}{2}\left(\mu(V)-\mu(U)+1\right)\left(\mu(V)-\mu(U)+2\right)-\Delta(U)-\Delta(V)\right).
\end{split}
\end{align}

\medskip\noindent
Since the relative Euler characteristic depends only on the rank, slope, and discriminant of the corresponding sheaves, it induces a bilinear pairing $(\xi,\zeta):=\chi(\xi^*,\zeta) = \chi(\xi \otimes \zeta)$ in the Grothendieck group $K_0(\PP^2)\otimes \mathbb{R}\cong \mathbb{R}^3$.  
We will use $\chi(\xi^*,\zeta)$ and $\chi(\xi \otimes \zeta)$ interchangeably, but take the convention of using $\chi(\xi \otimes \zeta)$ when considering interpolation and $\chi(\xi^*,\zeta)$ otherwise. 
Using this pairing, we may consider the set of characters which are orthogonal to the character $\xi$, which we denote by $\xi^{\perp}$. 
Denoting the moduli space of coherent sheaves with Chern character $\xi$ by $M(\xi)$, the Picard group of $M(\xi)$ can be identified with $\xi^{\perp}$ as long as $\Delta(\xi)$ is sufficiently large.
More precisely, if $\Delta(\xi)$ is above the Dr\'ezet-Le Potier curve, see \cite{DL}, then
\[\mathrm{Pic}(M(\xi))\cong \{\zeta \in K_0(\PP^2)\ | \ (\xi,\zeta)=0 \}.\]

\medskip\noindent
As orthogonality does not depend on the rank of the bundle, but only on its slope and discriminant, we can then visualize $\xi^{\perp}$ in the $(\mu,\Delta)$-plane, where it forms a parabola.

\subsection{Exceptional bundles}\label{Exceptional bundles}
\medskip\noindent
The geometry of coherent semi-stable sheaves on the plane is largely controlled by exceptional vector bundles.
See \cite{LP} or \cite{DL} for a more thorough introduction to exceptional bundles.

\medskip\noindent
A vector bundle $E$ is called \textit{exceptional} if $\mathrm{ext}^i(E,E) =0$ for $i>1$ and $E$ has only homotheties as automorphisms. 
In particular, this implies that $\chi(E,E)=1$ and consequently, by \eqref{CHI} above, its discriminant is \[\Delta(E) = \frac{1}{2}\left(1-\frac{1}{\mathrm{rk}(E)^2}\right).\]
It is clear from this equation that the discriminant of an exceptional bundle is less than $\frac{1}{2}$.
In fact, exceptional bundles are the only stable sheaves with discriminant less than $\tfrac{1}{2}$.
In addition, an exceptional bundle with slope $\alpha$ is unique up to isomorphism, therefore we can refer to \textit{the} exceptional bundle $E_{\alpha}$ with slope $\alpha$, which also has discriminant $\Delta_\alpha$.
The line bundles $\mathcal{O}_{\PP^2}(d)$, as well as the tangent bundle $\mathcal{T}_{\PP^2}$, are examples of exceptional bundles.

\medskip\noindent
There is a bijection from the dyadic integers to the set of slopes of exceptional bundles, $\epsilon:\mathbb{Z}\left[\frac{1}{2}\right] \to \mathfrak{E}$, defined by the rules: 
\begin{equation*}
\begin{aligned}
1. &\quad  \epsilon(n) = n\text{ for }n\in\mathbb{Z},\text{ and}\\[2mm]
2. & \quad \epsilon\left(\frac{2p+1}{2^q}\right) = \epsilon\left(\frac{p}{2^{q-1}}\right)\cdot\epsilon\left(\frac{p+1}{2^{q-1}}\right), 
&\text{ where by definition } \alpha\cdot\beta := \frac{\alpha+\beta}{2}+\frac{\Delta_\beta-\Delta_\alpha}{3+\alpha-\beta}.
\end{aligned} 
\end{equation*}

\medskip\noindent
Exceptional bundles on the plane sit in \textit{(complete strong) exceptional collections} which are triples of the form $\{E_\alpha,E_{\gamma},E_\beta\}$ where the only nonzero $\mathrm{Ext}$ groups between any pair of them are $\Hom(E_\alpha,E_{\gamma})$, $\Hom(E_\alpha,E_\beta)$, and $\Hom(E_{\gamma},E_\beta)$.
We will focus on the exceptional collections where $\alpha = \epsilon\left(\frac{p}{2^q}\right)$, $\beta = \epsilon\left(\frac{p+1}{2^q}\right)$, and $\gamma=\alpha\cdot\beta$. 
Any triple of slopes of that form is an exceptional collection.

\medskip\noindent
Each exceptional bundle $E$ determines two parabolas in the $(\mu,\Delta)$-plane defined by $\chi(*,E) =0$ and $\chi(E,*)=0$. 
The equation of each of these parabolas is determined by \eqref{CHI} above. We are interested in the intersection points of these curves with the line $\Delta = \frac{1}{2}$ in a neighbourhood of $E$; we denote such points $G$ and $F$.

\begin{definition}\label{Cherncharacter}
For any log Chern character $\xi$, the $(\mu,\Delta)$-coordinates of $F$ and $G$ in the previous paragraph, together with any sufficiently divisible rank, are called the \textit{right endpoint log Chern character} and \textit{left endpoint log Chern character}, respectively. \end{definition}

\noindent
By forgetting the rank, we abuse notation by referring to \textit{the} right (respectively, left) endpoint Chern character.
In Figure 1, these have been marked as $F$ and $G$, respectively. For a fixed Chern character $\xi$, following \cite{CHW}, we will call an exceptional bundle $E$ the \textit{(primary) controlling exceptional} bundle of $\xi$, if the larger intersection point of $\xi^{\perp}$ and $\Delta = \frac{1}{2}$ lies between the left and right endpoint Chern characters of $E$. 
The secondary controlling exceptional bundle is defined analogously for the smaller intersection of $\xi^{\perp}\cap \{\Delta=\frac{1}{2}\}$, if the rank of $\xi$ is at least 3.

\begin{center}
\begin{figure}[htb]\label{FIG1}
\resizebox{.75\textwidth}{!}{\includegraphics{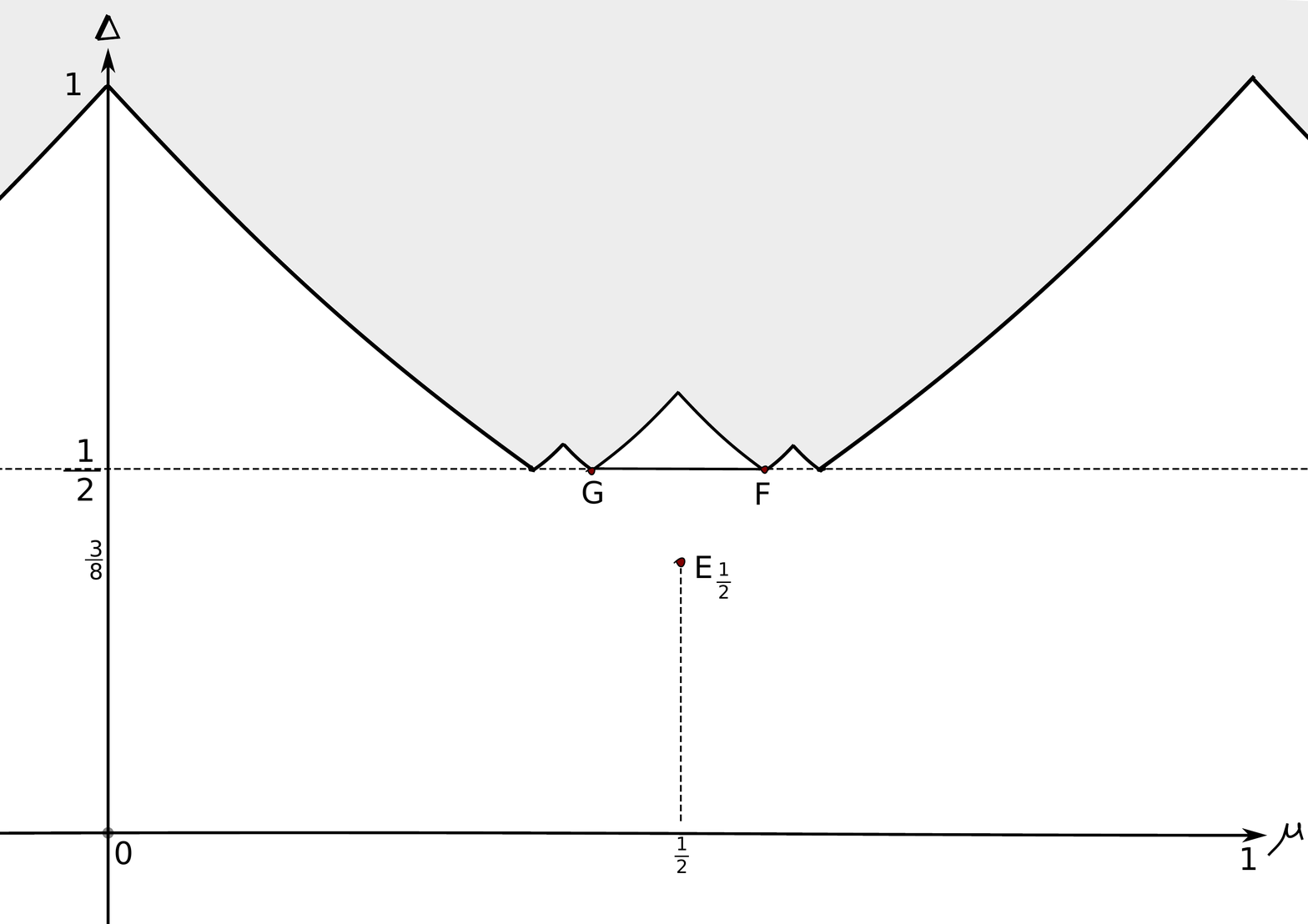}}
\caption{Endpoints $F,G$ prescribed by the Dr\'ezet-Le Potier curve over the exceptional bundle $E_{\frac{1}{2}}$.}
\end{figure}
\end{center}

\subsection{Free resolution of a coherent sheaf}\label{section2.4}

\medskip\noindent
The minimal free resolution (MFR) of a coherent sheaf on $\PP^2$ is determined by a matrix with homogeneous-polynomial entries. Indeed, any coherent sheaf $U$ on $\mathbb{P}^2$ has a minimal free resolution of the form \[0 \to \bigoplus_{i=1}^m \mathcal{O}_{\PP^2}(-d_{2,i})\overset{M}{\longrightarrow} \bigoplus_{i=1}^{m+k} \mathcal{O}_{\PP^2}(-d_{1,i}) \to U \to 0,\]
where $M$ is a matrix with homogeneous-polynomial entries. Recall that the \textit{Betti number} $\beta_{1,j}$ (respectively, $\beta_{2,j}$) is the number of summands in the middle term (respectively, left term) of the short exact sequence above such that $d_{1,i} = j$ (respectively, $d_{2,i} = j+1$). 
We will sometimes refer to these numbers as the \textit{left-hand side Betti numbers} making reference to the direction of the resolution.
For vector bundles, one can similarly take a free resolution of the bundle to the right by dualizing the minimal free resolution of the dual bundle. 
More precisely, let $U$ be a vector bundle of rank $\geq 2$. Its dual $U^*$ has a minimal free resolution
\[0 \to \bigoplus_{i=1}^m \mathcal{O}_{\PP^2}(d_{1,i})\to \bigoplus_{i=1}^{m+k} \mathcal{O}_{\PP^2}(d_{2,i}) \to U^* \to 0.\]
Note the difference in indexing and sign between this resolution and the resolution above of $U$. We make this convention aiming at simplifying subsequent definitions.
Since all of the factors are locally free, dualizing is exact and we have short exact sequence \[0 \to U \to \bigoplus_{i=1}^{m+k} \mathcal{O}_{\PP^2}(-d_{2,i})\to \bigoplus_{i=1}^{m} \mathcal{O}_{\PP^2}(-d_{1,i}) \to 0.\]
We define the \textit{right-hand side Betti numbers} $\gamma_{2,j}$ (respectively, $\gamma_{1,j}$) of $U$ to be  the number of summands in the middle (respectively, right) term of the short exact sequence above such that $d_{2,i} = j+1$ (respectively, $d_{1,i} = j$). Note also that the left-hand side Betti numbers of $U^*$, $\beta_{i,j}$, and the right-hand side Betti numbers of $U$, $\gamma_{i,j}$, are related by the equation $\beta_{i,j} = \gamma_{i,-j}$.  In the case that $U^*$ is not locally free, we define these similarly, ignoring the torsion which appears in the dualized sequence. We will say that a resolution $\mathcal{R}^\bullet$ is \textit{contained} in a resolution $\mathcal{S}^\bullet$ if they both have the same (co)domain and, up to row and column reduction, the matrix of the map in $\mathcal{R}^\bullet$ is contained in the matrix of the map in $\mathcal{S}^\bullet$.

\medskip\noindent
We recall that one may encode the Betti numbers of the minimal free resolution in a diagram; namely the Betti diagram. We will follow the notation in \cite{E}.

\begin{center}
\begin{tabular}{ c } 
 Betti diagram of a sheaf $U$\\
\hline
\begin{tabular}{ c| c c   }
0 &   1 & 2     \\ \hline 
1 & $\beta_{1,1}$ & $\beta_{2,1}$     \\
$\vdots$ &  &   \\
$d-1$ & $\beta_{1,d}$ & $\beta_{2,d}$  \\
\end{tabular}\\ 
\end{tabular}
\end{center}

\medskip\noindent
If a sheaf $U$ is general in its moduli space $M(\xi)$, then its minimal free resolution is the well-known \textit{Gaeta resolution}.
If $\chi\left(\mathcal{T}_{\mathbb{P}^2}(-d-1), U\right) \geq 0$, where $d$ is the minimal integer such that $h^0(\mathbb{P}^2,U(d))>0$ and $\mathcal{T}_{\mathbb{P}^2}$ denotes the tangent sheaf of $\PP^2$, then the Gaeta resolution of $U$ is 
\[0 \to \mathcal{O}_{\PP^2}(-d-2)^{-\chi\left(\mathcal{O}_{\PP^2}(-d+1),U\right)}  \oplus \mathcal{O}_{\PP^2}(-d-1)^{\chi\left(\mathcal{T}_{\mathbb{P}^2}(-d-1), U\right)} \to   \mathcal{O}_{\PP^2}(-d)^{\chi\left(\mathcal{O}_{\PP^2}(-d),U\right)} \to U \to 0.\]
Similarly, if $\chi\left(\mathcal{T}_{\mathbb{P}^2}(-d-1), U\right) \leq 0$, then the Gaeta resolution of $U$ is \[0 \to \mathcal{O}_{\PP^2}(-d-2)^{-\chi\left(\mathcal{O}_{\PP^2}(-d+1),U\right)}  \to \mathcal{O}_{\PP^2}(-d-1)^{-\chi\left(\mathcal{T}_{\mathbb{P}^2}(-d-1), U\right)} \oplus   \mathcal{O}_{\PP^2}(-d)^{\chi\left(\mathcal{O}_{\PP^2}(-d),U\right)} \to U \to 0.\]
A Gaeta minimal free resolution is said to be \textit{pure} if only two line bundles appear in it.

\medskip\noindent
The general object of the moduli space $M(\xi)$ admits another resolution in terms of exceptional bundles, which are usually of rank higher than 1 \cite[Proposition 5.3]{CHW}. 
Indeed, if $E_{\alpha\cdot\beta}$ is the controlling exceptional bundle of $\xi$ and $\chi(E_{-(\alpha\cdot\beta)}, U) > 0$ for a general sheaf $U \in M(\xi)$, then $U$ has resolution 
\begin{equation}\label{res1}
0 \to \left(E_{-\alpha-3}\right)^{-\chi\left(E_{-\alpha},U\right)}   \to\left(E_{-\beta}\right)^{-\chi\left(E_{-(\alpha\cdot(\alpha\cdot\beta))},U\right)}  \oplus \left(E_{-(\alpha\cdot\beta)}\right)^{\chi\left(E_{-(\alpha\cdot\beta)},U\right)} \to U \to 0.
\end{equation}

\medskip\noindent
Similarly,  if $E_{\alpha\cdot\beta}$ is the controlling exceptional bundle of $\xi$ and $\chi\left(E_{-(\alpha\cdot\beta)}, U\right) \leq 0$ for a general sheaf $U \in M(\xi)$, then $U$ has resolution
\begin{equation}\label{res2}
0 \to \left(E_{-3-(\alpha\cdot\beta)}\right)^{-\chi\left(E_{-(\alpha\cdot\beta)},U\right)} \oplus \left(E_{-\alpha-3}\right)^{\chi\left(E_{-((\alpha\cdot\beta)\cdot \beta)},U\right)}   \to\left(E_{-\beta}\right)^{-\chi\left(E_{-\beta},U\right)}  \to U \to 0.
\end{equation}
We have used relative Euler characteristics to express the exponents above as they come from the Beilinson spectral sequence. This notation differs from that of \cite{CHW} but they both denote the same quantities. Also notice that, in case $E_{\alpha\cdot \beta}$ is a line bundle, then the resolutions above coincide with the classical Gaeta minimal free resolution. Our main theorem will also incorporate this case.

\medskip\noindent
Each of the vector bundle resolutions (\ref{res1}) and (\ref{res2}) can be written in the derived category $D^b(\PP^2)$ as an exact triangle, called \textit{Gaeta triangle}:
\begin{equation}\label{kfu}
E_{-(\alpha\cdot\beta)}^{m}\to U \to W[1] \to \cdot,
\qquad
W \to U \to E_{-(\alpha\cdot\beta)-3}^{-m}[1]\to \cdot, \qquad\text{ or } \qquad E^n_{-\beta} \rightarrow U\rightarrow E^\ell_{-\alpha-3}[1] \rightarrow \cdot\end{equation}
where $m$, $n$, and $\ell$ are the appropriate powers and $W$ denotes a remainder complex. The triangle on the right-hand side, whose terms are $E_{-\beta}$ and $E_{-\alpha -3}[1]$, corresponds to the special case when $\chi(E_{-(\alpha\cdot \beta)},U)=0$. We will use these triangles in Section \ref{sec: eff} in order to compute the effective cone of divisors of the moduli space $M(\xi)$, and in Section \ref{sec: mov} in computing the cone of movable divisors of the Hilbert scheme of $n$ point on $\PP^2$, for two infinite families of values of $n$.

\subsection{Birational geometry}
We want to use the minimal free resolutions of sheaves to study the birational geometry of their moduli spaces. Let us briefly recall the notions at stake in birational geometry and refer the reader to \cite{Lazarsfeld} for a more thorough treatment of those.

\medskip\noindent
Recall that the N\'eron-Severi space $\mathrm{NS}(X)_\mathbb{Q}$ of a smooth algebraic variety $X$ is the vector space over $\mathbb{Q}$ of divisors modulo numeric equivalence. Inside the N\'eron-Severi space, the closure of the convex cone spanned by the classes of subvarieties is the \textit{effective cone}, $\mathrm{Eff}(X)$.
The effective cone decomposes according to the positivity of the divisors; specifically, according to their stable base locus.
The \textit{stable base locus} of a divisor $D$ on $X$, denoted $\mathbf{B}(D)$, is the intersection of all divisors linearly equivalent to any positive integer multiple of $D$, \[\mathbf{B}(D):= \bigcap_{m>0}\bigcap_{D' \in \vert mD \vert}D'.\]
Moduli spaces of sheaves on $\mathbb{P}^2$ are Mori dream spaces \cite{CHW} which implies that there is an open subset of the N\'eron-Severi space where the stable base locus is locally constant \cite{ELMNP}. 
The convex cones where it is locally constant are the \textit{chambers} of the stable base locus decomposition. We refer to the union of such chambers of $\EFF(X)$, which we abbreviate as SBLD, as the stable base locus decomposition.


\section{Base locus of the edge of $\EFF(M(\xi))$ via free resolutions}
\label{sec: eff}

\medskip\noindent
In this section we prove Theorem A, which asserts that the Gaeta minimal free resolution determines the effective cone $\EFF(M(\xi))$.  
This will be the first step of a program, which uses minimal free resolutions of sheaves on $\mathbb{P}^2$, to recover the Bridgeland destabilizing objects of those sheaves and hence the stable base locus decomposition of their moduli space.
In order to do this first step, we have to examine the properties of the maps that occur in the minimal free resolutions.

\medskip\noindent
We call a matrix with homogeneous-polynomial entries a  $\zeta$-\textit{map} if it can be row-and-column reduced to the  following block-form
\[M=\left(\begin{array}{c|c}
0  & A(M) \\ \hline
 B(M) & C(M)
\end{array}\right),\]
where $B(M)$ can be realized as the map of the minimal free resolution of a sheaf with Chern character $\zeta$. Note, the matrix $B(M)$ might be empty when $\zeta$ is the character of a line bundle. Notice that if the matrices $M$ and $B$ come from minimal free resolutions, then they are full-rank matrices and it follows that $A$ is a full rank matrix as well.

\medskip\noindent
Fix a Betti diagram, $G$. Let $S$ be the set of $\zeta$'s such that the general map in a resolution with Betti diagram $G$ is a $\zeta$-map. 
Then we say a map $M$ is a \textit{general map with Betti diagram $G$} if the set of $\zeta$'s for which it is a $\zeta$-map is contained in $S$.

\medskip\noindent
We have defined $\zeta$-maps in order to differentiate between two sheaves whose minimal free resolutions both have the same Betti diagram but whose maps have different properties. This definition emphasizes that a minimal free resolution of a sheaf consists of a Betti diagram and a map.

\begin{definition}\label{generalgaeta}
Let $U$ be a coherent sheaf in $M(\xi)$ with minimal free resolution
\[0\rightarrow \bigoplus^r_{i=1}  \mathcal{O}_{\PP^2}(a_i) \xrightarrow{M} \bigoplus^s_{i=1}  \mathcal{O}_{\PP^2}(b_i)\rightarrow U\rightarrow 0\] such that $a_i\le a_{i+1}$ and $b_i\le b_{i+1}$. Then $U$ is said to be  $\zeta$-\textit{admissible} (respectively, \textit{general admissible}) if the matrix $M$ is a $\zeta$-map (respectively, general map with that Betti diagram). \end{definition}

\medskip\noindent
If the coherent sheaf $U\in M(\xi)$ has a Gaeta minimal free resolution and is $\zeta$-admissible (respectively, general admissible), 
then $U$ will be called \textit{Gaeta} $\zeta$-\textit{admissible} (respectively, \textit{Gaeta admissible}). 
The former label (Gaeta) tells us the Betti diagram of the general sheaf with this property while the latter label (($\zeta$-)admissible) tells us a property of the map.




\medskip\noindent
The Gaeta triangle in which a sheaf $U$ fits depends on the sign of the Euler characteristic $\chi(E_{-(\alpha\cdot\beta)}, U)$, which means that our statement and proofs must depend on it as well. Let us restate Theorem A using the definitions above. 
It is worth mentioning that the matrices $A$ or $B$ in the following Theorem can have zero size. This is the case  when the Gaeta triangle involves line bundles in its factors.

\begin{theorem}\label{MAINdetailed} 
Let $U\in M(\xi)$ be a sheaf and let $\alpha\cdot \beta$ be the controlling exceptional slope of $\xi$.  Then there are three options:
\begin{enumerate}
\item[(a)] Assume $\chi(E_{-(\alpha\cdot\beta)}, U) > 0$. 
Then, the sheaf $U$ admits a resolution as in \eqref{res1} if and only if $U$ is Gaeta $\zeta$-admissible, where $\zeta$ is the character of $E_{-(\alpha \cdot \beta)}^{\chi(E_{-(\alpha \cdot \beta)},U)}$.

\medskip
\item[(b)] Assume $\chi (E_{-(\alpha\cdot\beta)}, U)< 0$. Then, the sheaf $U$ admits a resolution as in \eqref{res2} if and only if  $U$ is Gaeta $\zeta$-admissible where $\zeta$ is the character of a sheaf $W$ defined by the short exact sequence \[0 \to E_{-\alpha-3}^{-\chi\left(E_{-((\alpha \cdot \beta)\cdot \beta)},U\right)}\to E_{-\beta}^{\chi(E_{-\beta},U)}\to  W \to 0.\]

\medskip
\item[(c)] Assume $\chi (E_{-(\alpha\cdot \beta)}, U)= 0$. Then, the sheaf $U$ admits the following resolution \begin{equation}\label{res3}
0 \to (E_{-3-\alpha})^{-\chi\left(U\otimes E_{\alpha}\right)}   \to \left(E_{-\beta}\right)^{-\chi(U\otimes E_{\beta})}  \to U \to 0.
\end{equation}
if and only if $U$ is Gaeta $\zeta$-admissible where $\zeta$ is the character of $E_{-\beta}^{-\chi(U\otimes E_{\beta})}$.
 \end{enumerate}

\end{theorem}





\medskip\noindent
We prove this by showing that any general Gaeta resolution is the ``sum'' of free resolutions of the two objects which form a Gaeta triangle in ($\ref{kfu}$). In order to do that, first we will show that the Gaeta resolution has the same Betti numbers as the direct sum of the (general) Betti numbers of each of the terms in the Gaeta triangle in which it fits. Secondly, we show that the resulting map in the free resolution is a $\zeta$-map and that this free resolution is minimal.  Subsections \ref{Positivecase}, \ref{NegativeCase}, and \ref{sec: exceptional case} prove parts (a), (b), and (c) of the theorem, respectively.

\begin{definition}\label{DDGaeta}
The complement of the open set in $M(\xi)$ which parametrizes Gaeta admissible sheaves has codimension bigger or equal than 1, and we denote it by \[D_{\tiny{Gaeta}}=\{F\in M(\xi)\ | \ F \text{ is not a Gaeta admissible sheaf} \}.\]
\end{definition}

\medskip\noindent
The following three corollaries study the relationship between the stable base locus of $D_V$, the edge of $\EFF(M(\xi))$, and $D_{\tiny{Gaeta}}$.

\begin{cor}\label{cor: EffRigid}
Let $U\in M(\xi)$ be a sheaf and let $\alpha\cdot \beta$ be the controlling exceptional slope of $\xi$.

\begin{itemize}
    \item[(a)]  If $D_V$ spans the primary extremal ray of $\EFF(M(\xi))$, then we have that the stable base locus of $D_V$ satisfies $\textbf{B}({D_V})\subset D_{\tiny{Gaeta}}$. 
    
    \item[(b)] Assume that $\chi(E_{-(\alpha\cdot \beta)},U)=0$, then $D_{\tiny{Gaeta}}\subset M(\xi)$ is an irreducible, reduced divisor and it spans the primary extremal ray of $\mathrm{Eff}(M(\xi))$.  
Moreover, this divisor is the base locus of the primary extremal chamber of $\mathrm{Eff}(M(\xi))$.
\end{itemize}

\end{cor}

\begin{proof}
(a) In order to show this, it suffices to show that if a sheaf $U$ is Gaeta admissible, it is not in the base locus of $D_V$, where $D_V$ spans the primary extremal ray of the effective cone.
If $U$ is Gaeta admissible, then admits a resolution of type \eqref{res1} or \eqref{res2} so it not in some Brill-Noether divisor with class $D_V$. 
Therefore, it is not in the stable base locus of $D_V$.

\medskip\noindent
(b) By \cite[Proposition 5.9]{CHW}, the locus of sheaves which do not have a Gaeta triangle is a divisor, $D_{E_{\alpha.\beta}}$, which spans the primary edge of the effective cone $\EFF(M(\xi))$.
By the converse of \cite[Corollary 7.2]{CHW}, this divisor is the base locus of the primary extremal chamber.
Note, \cite{CHW} does not state the converse of their Corollary 7.2, however it does hold as they indirectly note in the discussion after Theorem 1.3.
We can argue such a converse as follows. Let $\pi:M(\xi) \dashrightarrow Kr_N(m_1,m_2)$ be the rational map to the moduli space of Kronecker modules induced by the association $U\mapsto W$, where $W$ is a complex of the form \cite[Prop. 6.3]{CHW} $$E_{-\alpha-3}^{m_1}\rightarrow E_{-\beta}^{m_2}.$$
If a sheaf $U$ is in the divisor $D_{E_{\alpha.\beta}}$, then it either does not have a corresponding Kronecker module (it is in the indeterminacy locus of $\pi$), or it is in one of the positive dimensional fibers of $\pi$. Indeed,
if $U\in D_{E_{\alpha \cdot\beta}}$ and has a corresponding Kronecker module, then the Beilinson spectral sequence degenerates to
\[0 \to E_{-\alpha-3}^a \oplus E_{-(\alpha.\beta)}^b \to E_{-(\alpha.\beta)}^b \oplus E_{-\beta}^c \to U \to 0\]
with $b>0$.
The fibers of $\pi$ correspond to fixing the Kronecker module $i.e.$, the map $E_{-\alpha-3}^a  \to E_{-\beta}^c$, and  varying the maps to and from $E_{-(\alpha.\beta)}^b$. 
When $b>0$, this means $\pi$ has positive dimensional fibers over the Brill-Noether divisor $D_{E_{\alpha \cdot \beta}}$; hence it gets contracted and constitutes the base locus.
By Theorem \ref{MAINdetailed}, part (c), a sheaf $U\in M(\xi)$ has a Gaeta triangle if and only if it is $\zeta$-admissible, with $\zeta$ the character of $E_{-\beta}^{-\chi(U\otimes E_{\beta})}$, therefore $D_{E_{\alpha \cdot \beta}}=D_{\tiny{Gaeta}}$. 
Moreover, \cite[Theorem 7.3]{CHW} yields $D_{\tiny{Gaeta}}$ irreducible and reduced. 
\end{proof}

\medskip\noindent
There are cases for which the inclusion in Part (a) is strict. For example, consider the Hilbert scheme of 9 points on the plane, and observe that the locus of 9 points which are a complete intersection of two cubics is contained in $D_{\tiny{Gaeta}}$ but fails to be in $\textbf{B}({D_V})$.

\begin{definition}
We call the Chern charter $\xi$ \textit{rigid} if $D_{\tiny{Gaeta}}$ has codimension $1$.
\end{definition}

\medskip\noindent
There are two reasons for a sheaf to be contained in $D_{\tiny{Gaeta}}$. One corresponds to failing to be $\zeta$-admissible, hence having a special map. 
The other corresponds to having a Betti diagram which is not that of Gaeta. 
Among those, a \textit{divisorial Gaeta resolution} will be a minimal free resolution whose Betti diagram is not that of Gaeta but where the locus of sheaves with that resolution form a divisor in $M(\xi)$. 
The following two cases exhibit divisorial Gaeta resolutions:

\begin{enumerate}[1.]
    \item If the moduli space $M(\xi)$ is non-empty and the general $U\in M(\xi)$ has Gaeta resolution
    \begin{equation}\label{pure1}
    0\to\OO_{\PP^2}(-1)^s\to\OO_{\PP^2}^t\rightarrow U\rightarrow 0,
    \end{equation}
    then a divisorial Gaeta resolution is 
    $$0\to\OO_{\PP^2}(-2)\oplus\OO_{\PP^2}(-1)^{s-3}\to\OO_{\PP^2}^{t-3}\oplus\OO_{\PP^2}(1)\rightarrow U'\rightarrow 0$$
    as long as $t>s\geq3$. 
    Note that for $t\geq s+2$, the general sheaf in this moduli space is a Steiner vector bundle; whereas for $t=s+1$, it is a (shifted) ideal of $n$ points, with $n$ a triangular number.
    
    \item If the moduli space $M(\xi)$ is non-empty and the general $U\in M(\xi)$ has Gaeta resolution
    \begin{equation}\label{pure2}
    0\to\OO_{\PP^2}(-2)^s\to\OO_{\PP^2}^t\rightarrow U\rightarrow 0    
    \end{equation}
    then a divisorial Gaeta resolution is
    $$0\to\OO_{\PP^2}(-2)^s\oplus\OO_{\PP^2}(-1)\to\OO_{\PP^2}(-1)\oplus\OO_{\PP^2}^t\rightarrow U'\rightarrow 0$$
    as long as $t> s\geq1$.
    \end{enumerate}

\medskip\noindent
\begin{rmk}\label{DivGaeta}
In the case of the Hilbert scheme of points on the plane, $M(1,0,-n)=\PP^{2[n]}$, the cases above account for all the divisorial Gaeta resolutions. We will analyze these two cases in Section \ref{sec: mov}.
\end{rmk}

\medskip\noindent
\begin{cor}\label{cor: DGaeta}
If the Gaeta resolution of $U\in M(\xi)$ is pure then the Chern character $\xi$ is rigid. Moreover, when the Gaeta resolution is pure as above in (\ref{pure1}), (\ref{pure2}), then a generic point in $D_{\tiny{Gaeta}}$ has the divisorial Gaeta minimal free resolution. 
\end{cor}
\begin{proof}
    If the Gaeta resolution of $U$ is pure, then we may apply Theorem \ref{MAINdetailed} part (c). 
    Indeed, in this case the controlling exceptional bundle $E_{(\alpha \cdot \beta)}$ is either a line bundle or the tangent bundle because the exponents in the Gaeta minimal free resolutions are defined using only those, and we have that $\chi(E_{-(\alpha \cdot \beta)}, U)=0$.

In both cases, the exceptional bundles $E_{-\alpha-3}$ and $E_{-\beta}$ are the line bundles that occur in the minimal free resolution of $U$. 
    By Corollary \ref{cor: EffRigid}, it follows that $D_{\tiny{Gaeta}}$ is an irreducible divisor (which spans an extremal ray of $\EFF(M(\xi))$). 

    Let $\mathcal{H}$ denote the family of sheaves above with \textit{divisorial Gaeta} resolution. This family is irreducible, is contained in $D_{\tiny{Gaeta}}$ and has codimension $1$ in $M(\xi)$, hence $\mathcal{H}=D_{\tiny{Gaeta}}$. 
    \end{proof}

\noindent
In the case of the Hilbert scheme of points $\PP^{2[n]}$, we have a converse of the previous Corollary.

\begin{cor}\label{cor: DGaeta2}
Let $n$ be such that $D_{\tiny{Gaeta}}$ has codimension 1 in $\PP^{2[n]}$. Then, the Gaeta minimal free resolution is pure or the generic point in $D_{\tiny{Gaeta}}$ has Gaeta Betti diagram. \end{cor}
\begin{proof}
If $D_{\tiny{Gaeta}}$ is divisorial, since there are no divisorial Betti diagrams when the Gaeta resolution is not pure, then the general element of $D_{\tiny{Gaeta}}$ must have the Gaeta Betti diagram.
\end{proof}

\medskip\noindent
The previous results show that minimal free resolutions contain the information of the first Bridgeland wall. In Section \ref{sec: mov}, we prove that the minimal free resolutions also contain the information of the movable cone of $\PP^{2[n]}$, the Hilbert scheme of $n$ points on $\PP^2$, when the Gaeta resolution of the generic point in $\PP^{2[n]}$ is pure. This will imply that the minimal free resolutions have the information of the second Bridgeland wall in these cases.


\medskip\noindent
We now proceed to prove Theorem \ref{MAINdetailed}. Let $\xi = (s, \mu_\xi, n)$ be a fixed stable log Chern character such that the moduli space $M(\xi)$ is positive dimensional with Picard rank two.
Recall that a general $U\in M(\xi)$  fits into one of the triangles in (\ref{kfu}) where $W$, $F$ are complexes of sheaves in the derived category $D^b(\PP^2)$. Note $F$ is always a sheaf and so is $W$ if its rank is at least $2$. 
We first aim to show that the Gaeta minimal free resolution can be recovered from the Gaeta triangle. 
Following the notation of \cite{CHW}, the proof of this depends on the sign of the $\chi (E_{-(\alpha\cdot\beta)}, U)$, where $\alpha\cdot\beta$ denotes the exceptional slope associated to $\xi$. We treat these cases separately. Let $E_{\alpha\cdot\beta} = E^*(d)$ where $\mu_E \in [0,1)$ and $E^*$ denotes the dual vector bundle of $E$.


\subsection{The positive case: $\chi(E_{-(\alpha\cdot \beta)}, U)>0$.}\label{Positivecase}
We first consider the case where that Euler characteristic is positive. This case is labeled in \cite{CHW} as $(\xi,\xi_{\alpha\cdot\beta})>0$. Throughout this subsection we denote the Gaeta triangle of $U$, induced by (\ref{res1}), as 
\begin{equation*}
F\rightarrow U\rightarrow W[1] \rightarrow \cdot, \end{equation*}
where $F=E_{-(\alpha \cdot \beta)}^{\chi\left(E_{-(\alpha \cdot \beta)},U\right)}$ and $W$ is the complex \[E_{-\alpha-3}^{-\chi\left(E_{-((\alpha \cdot \beta)\cdot \beta)},U\right)}\to E_{-\beta}^{\chi(E_{-\beta},U)}\] sitting in degrees $-1$ and $0$. We need the following technical lemma.

\begin{lem}\label{INEQUALITY1}
Let $G$ be the left endpoint logarithmic Chern character of $E_{\alpha \cdot \beta} = E^*(d)$ as in Definition (\ref{Cherncharacter}).
Then $\chi(G^*,\xi) < 0$ if and only if the following inequality holds:
\[\frac{1}{2} d^2 +  \left(\sqrt{2 \Delta_E + \frac{5}{4}} - \mu_E+\mu_U\right)d + \frac{1}{2}\mu_E^2-\mu_E\left( \sqrt{2 \Delta_E + \frac{5}{4}}+\mu_U\right)  +  \Delta_E+\frac{1}{2}\mu_U^2 +\mu_U\sqrt{2 \Delta_E + \frac{5}{4}}<n. \]
\end{lem}

\begin{proof}
The log Chern character of $G$ is defined in the $(\mu,\Delta)$-plane by the two equations $\Delta = \frac{1}{2}$ and $\chi\left( E_{\alpha \cdot \beta},G\right) = 0$, where $E_{\alpha\cdot\beta} = E^*(d)$. These two equations can be combined to get 
\begin{equation*}
    \begin{aligned}
\frac{1}{2}(\mu_G+\mu_E-d+1)(\mu_G+\mu_E-d+2)-\Delta_E =\frac{1}{2},       
    \end{aligned}
\end{equation*}
which implies that $\mu_G =  d-\mu_E - \frac{3}{2} + \sqrt{2 \Delta_E + \frac{5}{4}}$.
Since we have that $\chi(G^*,U) <0$ if and only if 
\[\frac{1}{2}(\mu_U+\mu_G+1)(\mu_U+\mu_G+2)-\frac{1}{2}-n<0,\]
then the lemma follows when substituting the value of $\mu_G$ into the previous inequality. 
\end{proof}

\begin{prop}
\label{lem: pos Betti}
Let $U\in M(\xi)$ be a general element with associated controlling exceptional slope $\alpha\cdot\beta$. If $\chi(E_{-(\alpha\cdot \beta)}, U) > 0$, then the Betti numbers of $U$ are the sum of the left-hand side Betti numbers of $F$ and the right-hand side Betti numbers of $W$.  
\end{prop}
\begin{proof}
We first work out the case when $W$ is a complex of rank at least 2.  This implies $W$ can be assumed general in its moduli space as it fits into a resolution of type \eqref{res1} or \eqref{res2}.
In particular, it is a vector bundle with a Gaeta minimal free resolution. Likewise, $F$ is general sheaf in its moduli space.

Let $\alpha_{i,j}$ and $\beta_{i,j}$ be the left-hand side Betti numbers of $U$ and $F$, respectively.
Similarly, let $\gamma_{i,j}$ be the right-hand side Betti numbers of $W$; these numbers are equal to the left-hand side Betti numbers of $W^*$. 

\medskip\noindent
Since $U$ is general, it has Gaeta's resolution and it has at most three nonzero left-hand side Betti numbers, either $\{\alpha_{1,d}, \alpha_{1,d+1}, \alpha_{2,d+1}\}$ or $\{\alpha_{1,d}, \alpha_{2,d}, \alpha_{2,d+1}\}$.
Let us work out the former case, $\{\alpha_{1,d},\alpha_{1,d+1},\alpha_{2,d+1}\}$, and comment on the latter afterward. 
In this case, $\mu_E \in [0,\frac{1}{2})$ and it follows from the Beilinson's spectral sequence, with respect to $\{\mathcal{O}_{\PP^2}(-d),E_{\frac{1}{2}}(-d),\mathcal{O}_{\PP^2}(-d+1)\}$, that these numbers are:  
\[\alpha_{1,d} = \chi\left(\mathcal{O}_{\PP^2}(-d),U\right),\quad \alpha_{1,d+1} = -\chi\left(E_{\frac{1}{2}}(-d),U\right), \quad \alpha_{2,d+1} = -\chi\left(\mathcal{O}_{\PP^2}(-d+1),U\right).\]

\medskip\noindent
We then define: \[\beta_{1,d} = \chi\left(\mathcal{O}_{\PP^2}(-d),F\right), \quad \beta_{1,d+1} = -\chi\left(E_{\frac{1}{2}}(-d),F\right),\quad \beta_{2,d+1} = -\chi\left(\mathcal{O}_{\PP^2}(-d+1),F\right),\]
\[\gamma_{1,d} = -\chi\left(\mathcal{O}_{\PP^2}(d+3),W^*\right),\quad \gamma_{1,d+1} = \chi\left(E_{\frac{1}{2}}(d+2),W^*\right), \quad \gamma_{2,d+1} = \chi\left(\mathcal{O}_{\PP^2}(d+2),W^*\right).\]
There is no conflict in notation as after we have shown these have the correct sign, it follows that those are in fact the (right-hand side) Betti numbers. 

\medskip\noindent
We must check that these Euler characteristics sum correctly and have the correct sign. 
We first check that they sum correctly.
In other words, we aim to show that the following three equalities of non-negative integers hold:
\begin{align}\label{BETTI1}
\begin{split}
\alpha_{1,d}&=\beta_{1,d}+\gamma_{1,d},\\ \alpha_{1,d+1}&=\beta_{1,d+1}+\gamma_{1,d+1}, \\ 
\alpha_{2,d+1}&=\beta_{2,d+1}+\gamma_{2,d+1}.
\end{split}
\end{align}
Substituting the values in (\ref{BETTI1}), these equations give:
\begin{align}\label{INEQUALITIES1}
\begin{split}
\chi\left(\mathcal{O}_{\PP^2}(-d),U\right) &=  \chi\left(\mathcal{O}_{\PP^2}(-d),F\right)-\chi\left(\mathcal{O}_{\PP^2}(d+3),W^*\right), \\
-\chi\left(E_{\frac{1}{2}}(-d),U\right)    &= -\chi\left(E_{\frac{1}{2}}(-d),F\right)+\chi\left(E_{\frac{1}{2}}(d+2),W^*\right),\\
-\chi\left(\mathcal{O}_{\PP^2}(-d+1),U\right)    &= -\chi\left(\mathcal{O}_{\PP^2}(-d+1),F\right)+\chi\left(\mathcal{O}_{\PP^2}(d+2),W^*\right).
\end{split}\end{align}

\medskip\noindent
These three equations follow from the additivity of the relative Euler characteristic and Serre duality. Indeed, using the left-hand side triangle in (\ref{kfu}), we have that $\chi\left(V,U\right) =  \chi\left(V,F\right)-\chi\left(V,W\right)$ by the additivity of the relative Euler characteristic for all complexes of sheaves. By Serre duality, we have $\chi(V,W) = \chi(W,V \otimes \omega_{\mathbb{P}^2})$ for all locally free sheaves $V$ and $W$. This duality, and the definition of the relative Euler characteristic, imply that $\chi(W,V(-3)) = \chi(W^* \otimes V(-3)) = \chi(V^*(3),W^*)$. Substituting $\mathcal{O}_{\PP^2}(-d)$, $E_{\frac{1}{2}}(-d)$, and $\mathcal{O}_{\PP^2}(-d+1)$ in for $V$  gives the three equalities.

\medskip\noindent
Let us now examine the signs. Since $F = E(-d)^{\chi\left(E(-d),U\right)}$ by \cite[Theorem 5.7]{CHW}, we immediately have that 
\begin{equation*}
    \begin{aligned}
       \chi\left(\mathcal{O}_{\PP^2}(-d),F\right)\geq 0, \quad \chi\left(E_{\frac{1}{2}}(-d),F\right)\leq 0,\quad \chi\left(\mathcal{O}_{\PP^2}(-d+1),F\right) \leq 0.
    \end{aligned}
\end{equation*}

\medskip\noindent
Let us now examine the signs of the three Euler characteristics involving $W^*$. 
The Euler characteristics involving $W^*$ have the sign we claimed above if and only if the following three inequalities hold 
\begin{equation}\label{INN1}
\begin{aligned}
0 &\leq \chi\left(\mathcal{O}_{\PP^2}(-d),U\right) -\chi\left(\mathcal{O}_{\PP^2}(-d),E(-d)\right)\chi\left(E(-d),U\right),\\
0 &\geq \chi\left(E_{\frac{1}{2}}(-d),U\right)   -\chi\left(E_{\frac{1}{2}}(-d),E(-d)\right)\chi\left(E(-d),U\right),\\
0 &\geq\chi\left(\mathcal{O}_{\PP^2}(-d+1),U\right)  -\chi\left(\mathcal{O}_{\PP^2}(-d+1),E(-d)\right)\chi\left(E(-d),U\right).
\end{aligned}
\end{equation}

\medskip\noindent
Each of these inequalities follows from the fact that $\chi\left(G^*, U\right) \leq 0$, where $G$ is the left endpoint logarithmic Chern character for $E^*(d)$. Let us show that the first inequality above holds; we omit the proofs that the other two inequalities hold because they are similar and no difficulty arises.

\medskip\noindent
Let $\xi = (s, \mu_{\xi}, n)$ denote the Chern character of $U$; expanding the first inequality above, using the Riemann-Roch formula, we have that 
\[s\left(A-n\right)-\left(r_E(\frac{1}{2}(\mu_E+1)(\mu_E+2)-\Delta_E) \right)\cdot s\left(r_E(\frac{1}{2}(\mu_{\xi}-\mu_E+d+1)(\mu_{\xi}-\mu_E+d+2)-\Delta_E-n)\right) \geq 0,\]
where $A:=\binom{d+\mu_{\xi}+2}{2}$.

\medskip\noindent
Solving for $n$ gives
\[B:=\frac{\binom{d+\mu_{\xi}+2}{2}-\left(r_E\left(\frac{1}{2}(\mu_{\xi}-\mu_E+d+1)(\mu_{\xi}-\mu_E+d+2)-\Delta_E\right)\right)\chi(E)}{(1-r_E\chi(E))} \leq n.\]
Since $\chi(E_{-(\alpha\cdot\beta)}, U)>0$, which implies that $\chi(G^*,U)<0$, then applying Lemma \ref{INEQUALITY1} simplifies this expression to
\[ \frac{1}{2} d^2 +  \left(\sqrt{2 \Delta_E + \frac{5}{4}} - \mu_E+\mu_{\xi}\right)d + \frac{1}{2}\mu_E^2-\mu_E\left(\sqrt{2 \Delta_E + \frac{5}{4}}+\mu_{\xi}  \right) +  \Delta_E +\frac{1}{2}\mu_{\xi}^2 +\mu_{\xi}\sqrt{2 \Delta_E + \frac{5}{4}}\geq B.\]


\noindent
Simplifying yields
\[0\leq -(1-r_E\chi(E))\left( \frac{1}{2}\mu_E^2-\mu_E \sqrt{2 \Delta_E + \frac{5}{4}}  -\mu_E\mu_{\xi}+  \Delta_E+\frac{1}{2}\mu_{\xi}^2 +\mu_{\xi}\sqrt{2 \Delta_E + \frac{5}{4}}\right)\]\[+ 1-\chi(E)\chi(E(-3)) +\left(1-r\chi(E)\right)\frac{1}{2}(\mu_{\xi}+3)\mu_{\xi}+r\chi(E)\mu_{\xi}\mu_E\]\[+ \left(\left(\mu_{\xi}+\frac{3}{2}-r_E\chi(E)(\mu_{\xi}+\frac{3}{2}-\mu_E)\right)-(1-r_E\chi(E))\left(\sqrt{2 \Delta_E + \frac{5}{4}} - \mu_E+\mu_{\xi}\right) \right)d,\]
which is a linear inequality in $d$. Thus, it suffices to show that the coefficient of $d$ is positive and that the entire expression vanishes at a value less than 1.  In order to do this, observe that the coefficient of $d$ can be written in terms of $\Delta_E$ and $\mu_E$. Using the bounds $\frac{3}{8} \leq \Delta_E < \frac{1}{2}$ and $\frac{3-\sqrt{5}}{2} < \mu_E < \frac{1}{2}$, we then find no values for which the coefficient is negative. It remains to argue that the zero occurs when $d<1$. Indeed, such a zero occurs at $d = \mu_E-\mu_{\xi}$, which can be verified by substituting this value into the expression. 
This implies that for all $d>1$ the first inequality of \eqref{INN1} holds.

\medskip\noindent
The cases where $W$ is a complex of rank $0$ or $1$ can be treated  analogously by considering $W$ and $W^*$ as objects in the derived category. 

\medskip\noindent
Similarly, in the case that the three possibly nonzero left-hand side Betti numbers are $\alpha_{1,d}$, $\alpha_{2,d}$, and $\alpha_{2,d+1}$, the proof is entirely analogous with the signs on the middle inequality flipped ($\mu_E \in \left[\frac{1}{2},1\right)$).
\end{proof}

\medskip\noindent
The next proposition claims that we can recover the Gaeta minimal free resolution out of the Gaeta triangle of \cite{CHW,Hui14}.

\begin{prop}
\label{lem: pos matrix}
Assume $\chi\left(E_{-(\alpha\cdot \beta)}, U\right) > 0$ for $U\in M(\xi)$ and let us denote the Gaeta triangle of $U$, induced by (\ref{res1}), as 
\begin{equation*}
F\rightarrow U\rightarrow W[1] \rightarrow \cdot, \end{equation*}
where $F=E_{-(\alpha \cdot \beta)}^{\chi\left(E_{-(\alpha \cdot \beta)},U\right)}$ and $W$ is considered as a complex when it has rank 0 or 1; and as a vector bundle otherwise. Then a general $U$ has the Gaeta minimal free resolution:
\[0\rightarrow  \mathcal{O}_{\PP^2}(-d-2)^{j_1+j_2} \xrightarrow{M}  \mathcal{O}_{\PP^2}(-d-1)^{l_1+l_2} \oplus \mathcal{O}_{\PP^2}(-d)^{n_1+n_2} \rightarrow U \rightarrow 0\quad \text{ if } \mu_E\in \left[0,\frac{1}{2}\right),\text{ or}\]
\[0\rightarrow    \mathcal{O}_{\PP^2}(-d-2)^{j_1+j_2}\oplus \mathcal{O}_{\PP^2}(-d-1)^{-l_1-l_2} \xrightarrow{M}   \mathcal{O}_{\PP^2}(-d)^{n_1+n_2} \rightarrow U \rightarrow 0 \quad  \text{ if } \mu_E\in \left[\frac{1}{2},1\right),\]
where the matrix $M$ can in both cases be written in the form 
\[
M=\left(
\begin{array}{c|c}
0& A(M) \\ \hline
 B(M) & C(M)
\end{array}\right),
\]
the submatrix $B(M)$ is the map of the minimal free resolution of $E_{-(\alpha \cdot \beta)}^{\chi\left(E_{-(\alpha \cdot \beta)},U\right)}$, and
\begin{equation*}
    \begin{aligned}
 n_1=&-\chi(\mathcal{O}_{\PP^2}(d+3),W^*), &\quad n_2=&\chi(\mathcal{O}_{\PP^2}(-d),F),\\
l_1=&\chi(\mathcal{T}_{\PP^2}(d+1),W^*), &\quad l_2=& -\chi(\mathcal{T}_{\PP^2}(-d-1),F),\\
j_1=&\chi(\mathcal{O}_{\PP^2}(d+2),W^*), &\quad j_2=&-\chi(\mathcal{O}_{\PP^2}(-d+1),F).
    \end{aligned}
\end{equation*}
\end{prop}

\begin{proof}
Let $U$ be a general sheaf in $M(\xi)$ with the Gaeta triangle 
\[ F\overset{g}{\rightarrow} U\rightarrow W[1] \overset{f}{\rightarrow}\cdot, \] where 
$F=E_{-(\alpha \cdot \beta)}^{\chi\left(E_{-(\alpha \cdot \beta)},U\right)}$ and $W$ is the kernel of $g$ as a complex in the derived category $D^b(\PP^2)$.
Let us interpret the map $f$ as a map of complexes. This interpretation will yield the minimal free resolution of $U$ as its mapping cone. This will also reveal that the defining map $M$ is a $\zeta$-map, with $\zeta$ the character of $F$.

\medskip\noindent
Since $W$ is general in its moduli space, if it has rank at least two, then we may assume it is locally free (the cases in which the rank of $W$ is either zero or one are similar but with an extra torsion factor in the complex). Consequently, there is a dual locally free sheaf $W^*$.

\medskip\noindent
Thus, $F$ has a minimal free resolution \[F_{\bullet}: \quad 0 \to F_{-2} \overset{B}{\longrightarrow} F_{-1} \to F\to 0\] and $W^*$ has a minimal free resolution \[0 \to L_0^* \to L_{-1}^* \to W^*\to 0,\] where if $\mu_E \in \left[0,\frac{1}{2}\right)$, then
\begin{equation*}
\begin{aligned}
    F_{-2} &= \mathcal{O}_{\PP^2}(-d-2)^{-\chi(\mathcal{O}_{\PP^2}(-d+1),F)} \oplus \mathcal{O}_{\PP^2}(-d-1)^{\chi(\mathcal{T}_{\PP^2}(-d-1),F)},\\
    F_{-1} &= \mathcal{O}_{\PP^2}(-d)^{\chi(\mathcal{O}_{\PP^2}(-d),F)},\\
    L_{-1}^* &= \mathcal{O}_{\PP^2}(d+1)^{\chi(\mathcal{T}_{\PP^2}(d+1),W^*)} \oplus \mathcal{O}_{\PP^2}(d+2)^{\chi(\mathcal{O}_{\PP^2}(d+2),W^*)},\text{ and }\\
    L_0^* &= \mathcal{O}_{\PP^2}(d)^{-\chi(\mathcal{O}_{\PP^2}(d+3),W^*)},\\
\end{aligned}
\end{equation*}
and if $\mu_E \in \left[\frac{1}{2},1\right)$, then
\begin{equation*}
\begin{aligned}
    F_{-2} &= \mathcal{O}_{\PP^2}(-d-2)^{-\chi(\mathcal{O}_{\PP^2}(-d+1),F)},\\
    F_{-1} &= \mathcal{O}_{\PP^2}(-d)^{\chi(\mathcal{O}_{\PP^2}(-d),F)} \oplus \mathcal{O}_{\PP^2}(-d-1)^{-\chi(\mathcal{T}_{\PP^2}(-d-1),F)},\\
    L_{-1}^* &=  \mathcal{O}_{\PP^2}(d+2)^{\chi(\mathcal{O}_{\PP^2}(d+2),W^*)},\text{ and }\\
    L_0^* &= \mathcal{O}_{\PP^2}(d+1)^{-\chi(\mathcal{T}_{\PP^2}(d+1),W^*)} \oplus\mathcal{O}_{\PP^2}(d)^{-\chi(\mathcal{O}_{\PP^2}(d+3),W^*)}.\\
\end{aligned}
\end{equation*}

\medskip\noindent
Since $W$ is locally free, then $W$ fits into the following complex \[W_{\bullet}: \quad 0 \to W \to L_{-1} \overset{A}{\longrightarrow} L_0 \to 0.\]

\medskip\noindent
Now we can think of the map $f$ as a map of complexes
\begin{center}
    \begin{tikzcd}
        W_{\bullet}: \quad  0 \arrow[r] & L_{-1} \arrow{r}{A} \arrow{rd}{f} & L_0 \arrow{r}{d_0}  &  0 \\  
       F_{\bullet}: \quad      0 \arrow[r] & F_{-2} \arrow{r}{B}  & F_{-1} \arrow{r}{d_{-1}}  &  0, 
    \end{tikzcd}
\end{center}
where we interpret $f: ker(A)\rightarrow coker(B)$ in degree $-1$ and zero otherwise. Since $L_{-1}$ and $F_{-1}$ are sums of line bundles, the map $f$ is determined by a matrix $C$. We can now put these matrices together into another matrix $M$ as follows: \[M=\left(
\begin{array}{c|c}
0  & A \\ \hline
 B & C 
\end{array}\right).\]
Consequently there is a sheaf $U'$ which has the free resolution
\begin{equation}\label{eq: summand}
0\rightarrow L_{-1}\oplus F_{-2} \overset{M}{\longrightarrow} L_0\oplus F_{-1} \rightarrow U' \rightarrow 0.
\end{equation} 
It follows from Proposition \ref{lem: pos Betti} that this is the minimal free resolution of $U'$. Indeed, this follows from the fact that the matrices $A$, $B$ and $C$ are full rank. By taking the kernel of $A$ and cokernel of $B$, the matrix $M$ yields a vector bundle resolution of $U'$ which is exactly the same resolution as we started with, corresponding to $U$. 
Thus $U' = U$, and the free resolution \eqref{eq: summand} is the Gaeta resolution of $U$.
\end{proof}

\begin{cor}
With the notation of Proposition \ref{lem: pos matrix} the map $f$ yields a map between complexes $W_{\bullet}\overset{f}{\longrightarrow} F_{\bullet}$ and its mapping cone $M(f)_{\bullet}$ is the minimal free resolution of $U$ sitting in degrees $-2,-1,0$. Hence, the map $f$ induces an exact triangle in $D^b(\PP^2)$
\[W_{\bullet}\overset{f}{\longrightarrow} F_{\bullet}\to U_{\bullet}\to .\]
\end{cor}

\medskip
\subsection{The Negative Case: $\chi(E_{-(\alpha\cdot \beta)}, U ) < 0$}\label{NegativeCase}
The proofs in this subsection are entirely analogous to the previous ones and will have the same structure. We start with a technical lemma. Throughout this subsection we denote the Gaeta triangle of $U$, induced by (\ref{res2}), as 
\begin{equation*}
W\rightarrow U\rightarrow F[1] \rightarrow \cdot, \end{equation*}
where $F=E_{-(\alpha \cdot \beta)-3}^{-m}$ and $W$ is the sheaf defined by 
\[0 \to \left(E_{-\alpha-3}\right)^{\chi\left(E_{-((\alpha\cdot\beta)\cdot \beta)},U\right)}   \to\left(E_{-\beta}\right)^{-\chi\left(E_{-\beta},U\right)} \to W \to 0.\]

\begin{lem}\label{INEQUALITY1 Negative}
Let $G$ be the right endpoint logarithmic Chern character of $E_{\alpha \cdot \beta} = E^*(d)$.
Then $\chi(G^*,U) > 0$ if and only if the following inequality holds
\[\frac{1}{2} d^2 +  \left(3-\sqrt{2 \Delta_E + \frac{5}{4}} - \mu_E+\mu_U\right)d + \frac{1}{2}\mu_E^2-\mu_E\left(3- \sqrt{2 \Delta_E + \frac{5}{4}}+\mu_U\right)  \]\[ +  \Delta_E+\frac{1}{2}\mu_U^2 +\mu_U\left(3-\sqrt{2 \Delta_E + \frac{5}{4}}\right)>n. \]
\end{lem}
\begin{proof}
The proof is analogous to the proof of Lemma \ref{INEQUALITY1} but for the right endpoint logarithmic Chern character.
\end{proof}

\begin{prop}
Let $U\in M(\xi)$ be a general element with associated controlling exceptional slope $\alpha\cdot\beta$. If $\chi\left(E_{-(\alpha\cdot\beta)}, U\right) < 0$, then the left-hand side Betti numbers of $U$ are the sum of the left-hand side Betti numbers of $F$ and the right-hand side Betti numbers of $W$.  
\end{prop}

\begin{proof}
The proof is analogous to the proof of Proposition \ref{lem: pos Betti} where the key inequality utilizes the right endpoint, which makes the analogous inequalities to (\ref{INN1}) to hold once $d \geq -3+\mu_E - \mu_U$.
\end{proof}

\begin{prop}
\label{lem: neg matrix}
If $\chi\left(E_{-(\alpha\cdot\beta)}, U\right) < 0$ for $U\in M(\xi)$ and $U$ fits into the Gaeta triangle induced by (\ref{res2})
$$W\rightarrow U\rightarrow F[1] \rightarrow \cdot,$$
where $F=E_{-(\alpha \cdot \beta)-3}^{-m}$ and $W$ is the sheaf defined above, then a general $U$ has the Gaeta resolution 
\[0\rightarrow  \mathcal{O}_{\PP^2}(-d-2)^{j_1+j_2} \xrightarrow{M}  \mathcal{O}_{\PP^2}(-d-1)^{l_1+l_2} \oplus \mathcal{O}_{\PP^2}(-d)^{n_1+n_2} \rightarrow U \rightarrow 0 \text{ if } \mu_E\in \left(0,\frac{1}{2}\right] \text{ or } \]
\[0\rightarrow   \mathcal{O}_{\PP^2}(-d-2)^{j_1+j_2} \oplus \mathcal{O}_{\PP^2}(-d-1)^{-l_1-l_2} \xrightarrow{M}   \mathcal{O}_{\PP^2}(-d)^{n_1+n_2} \rightarrow U \rightarrow 0  \text{ if } \mu_E\in \left(\frac{1}{2},1\right],\]
where the matrix $M$ can be written as 
\[
M=\left(
\begin{array}{c|c}
0& A(M) \\ \hline
 B(M) & C(M)
\end{array}\right),
\]
where the submatrix $B(M)$ is the map of the minimal free resolution of $W$, the transpose of the submatrix $A(M)$ is the map of the minimal free resolution of $E_{-(\alpha \cdot \beta)-3}^{-\chi\left(E_{-(\alpha \cdot \beta)},U\right)}$, and the exponents are \begin{equation*}
    \begin{aligned}
 n_1=&-\chi(\mathcal{O}_{\PP^2}(d+3),F^*), &\quad n_2=&\chi(\mathcal{O}_{\PP^2}(-d),W),\\
l_1=&\chi(\mathcal{T}_{\PP^2}(d+1),F^*), &\quad l_2=& -\chi(\mathcal{T}_{\PP^2}(-d-1),W),\\
j_1=&\chi(\mathcal{O}_{\PP^2}(d+2),F^*), &\quad j_2=&-\chi(\mathcal{O}_{\PP^2}(-d+1),W).
    \end{aligned}
\end{equation*}
\end{prop}

\begin{proof}
The proof is entirely analogous to the proof of Proposition \ref{lem: pos matrix}.
\end{proof}

\medskip
\subsection{The exceptional case:}
\label{sec: exceptional case} $\chi(E_{-(\alpha\cdot\beta)}, U) = 0$. In this section the integer $d$ satisfies $E_{\alpha \cdot \beta} = E^*(d)$ and the proofs in this case are analogous to the proofs in the negative case after replacing $E_{\alpha\cdot\beta}$ with $E_{\alpha}$. In this case, the general sheaf $U\in M(\xi)$ admits the following resolution 
\begin{equation}\label{res3}
0 \to (E_{-3-\alpha})^{-\chi\left(U\otimes E_{\alpha}\right)}   \to \left(E_{-\beta}\right)^{\chi(U\otimes E_{\beta})}  \to U \to 0.
\end{equation}

\begin{prop}
Let $U\in M(\xi)$ be a general element with associated controlling exceptional slope $\alpha\cdot\beta$. If $\chi\left(E_{-(\alpha\cdot\beta)}, U\right) = 0$, then the Betti numbers of $U$ are the sum of the left-hand size Betti numbers of $F=\left(E_{-\beta}\right)^{\chi(U\otimes E_{\beta})}$ and the right-hand size  Betti numbers of $W=(E_{-3-\alpha})^{-\chi\left(U\otimes E_{\alpha}\right)}$. 
\end{prop}

\begin{proof}
This proof is entirely analogous to the proof of Proposition \ref{lem: pos Betti}. In this case, the key inequality considers the right endpoint, and the relevant inequalities hold once $d \geq -3+\mu_E - \mu_U$.
\end{proof}

\begin{prop}
\label{lem: exc matrix}
Let $U\in M(\xi)$ such that $\chi(E_{-(\alpha\cdot \beta)}, U) = 0$. Then a general $U$ has the following Gaeta resolution \[0\rightarrow  \mathcal{O}_{\PP^2}(-d-2)^{j_1+j_2} \xrightarrow{M}  \mathcal{O}_{\PP^2}(-d-1)^{l_1+l_2} \oplus \mathcal{O}_{\PP^2}(-d)^{n_1+n_2} \rightarrow U \rightarrow 0, \quad \text{ if } \quad \mu_E\in \left[0,\frac{1}{2}\right) \text{ or } \]
\[0\rightarrow   \mathcal{O}_{\PP^2}(-d-2)^{j_1+j_2} \oplus \mathcal{O}_{\PP^2}(-d-1)^{-l_1-l_2} \xrightarrow{M}   \mathcal{O}_{\PP^2}(-d)^{n_1+n_2} \rightarrow U \rightarrow 0,  \quad \text{ if } \quad \mu_E\in \left[\frac{1}{2},1\right),\]
where the matrix $M$ can be written in the following form 
\[
M=\left(
\begin{array}{c|c}
0& A(M) \\ \hline
 B(M) & C(M)
\end{array}\right),
\]
the submatrix $B(M)$ is the map of the minimal free resolution of $E_{-\beta}^{\chi(U\otimes E_{\beta})}$, the transpose submatrix $A(M)^t$ is the map of the minimal free resolution of $E_{-\alpha - 3}^{-\chi(U\otimes E_{\alpha})}$; the exponents are the same as in Prop \ref{lem: pos matrix}.
\end{prop}

\begin{proof}
The proof is entirely analogous to the proof of Proposition \ref{lem: pos matrix}.
\end{proof}

\medskip\noindent
Notice that either matrix $B$ or $A$ in the previous Proposition may be the zero-size matrix. This occurs when either $E_{-\beta}$ or $E_{-\alpha-3}$ is a line bundle. In case $E_{-\beta}$ and $E_{-\alpha-3}$ are both line bundles, then the resolution (\ref{res3}) coincides with Gaeta minimal free resolution.

\subsection{Proof of the Main Theorem} 
The previous subsections have shown that the Gaeta resolution can be recovered from the Gaeta triangle. We now finish the proof of Theorem \ref{MAINdetailed} with a uniform argument.

\begin{proof}[Proof of Theorem \ref{MAINdetailed}]
\medskip\noindent
Propositions \ref{lem: pos matrix}, \ref{lem: neg matrix}, and \ref{lem: exc matrix} show that if a sheaf fits into a Gaeta triangle 
\[F \to U \to W[1] \to \cdot \quad \]
with general $F$ and $W[1]$
then it has Gaeta minimal free resolution for which the resolution map $M$ is $\zeta$-map, with $\zeta$ the character of $F$.

\medskip\noindent
The converse holds. Indeed, if $U$ is a Gaeta admissible sheaf, then its minimal free resolution is $\zeta$-admissible. Then the sheaf fits into a triangle
\[F \to U \to W[1]\to \cdot\]
where $F\in M(\zeta)$ with general $F$ and $W[1]$. 
This is the Gaeta triangle of $U$.
\end{proof}

\subsection{The secondary edge of the effective cone}
The results of this section so far have only considered the primary edge of the effective cone. 
The case of the secondary edge is special for ranks $0$, $1$, and $2$; see \cite{CHW}.
For rank at least $3$, 
let us denote by $U^D$ the Serre dual of a sheaf $U$, i.e. $U^D = U^*(-3)$.

\begin{definition}
Let $D_{\tiny{Gaeta}^D}=\overline{\{U\in M(\xi)\ | \ U^D \text{ is not a Gaeta-admissible sheaf} \}}$. 
\end{definition}

\begin{cor}\label{thm: sec eff}
If $\xi^D$ is a rigid Chern character then $D_{\tiny{Gaeta}^D}\subset M(\xi)$ is a divisor and it spans the secondary extremal ray of $\mathrm{Eff}(M(\xi))$.
\end{cor}

\medskip\noindent
\textbf{Example:} Let us exemplify the results of this section for the log Chern character  $\xi=(3,\tfrac{2}{3},\tfrac{17}{9})$.
The controlling exceptional slope for the primary edge of $\EFF(M(\xi))$ is $\alpha\cdot \beta=0$ and  $\chi(E_{-(\alpha\cdot\beta)}, U) >0$ for $U\in M(\xi)$.
It follows by Theorem \ref{MAINdetailed}, part (a), that the resolution (\ref{res1}) of a general $U\in M(\xi)$ coincides with the Gaeta minimal free resolution
\[0\to \OO_{\PP^2}(-2)^4\to \OO_{\PP^2}(-1)^{6} \oplus \OO_{\PP^2} \to U\to 0.\]

\noindent
Hence, the the elements of $D_{\tiny{Gaeta}}$ will fail to have this minimal free resolution.

\medskip\noindent
Let us now describe the secondary edge of $\EFF(M(\xi))$. Consider the Serre dual of the previous log Chern character: $\xi^D=(3,-\tfrac{11}{3},-\tfrac{17}{9})$.
In this case, the controlling exceptional slope for the primary edge of $\EFF(M(\xi^D))$ is $\alpha \cdot \beta=\frac{22}{5}$ and $\chi(E_{-(\alpha\cdot\beta)}, U_D)=0$ for $U_D\in M(\xi^D)$. It follows that the Gaeta triangle of a general $U_D\in M(\xi^D)$ comes from the resolution (\ref{res3}):
\[0\to \OO_{\PP^2}(-7) \to E_{\frac{1}{2}}(-5)^{\oplus 2} \to U_D\to 0.\] 

\medskip\noindent
Using the notation of Proposition \ref{lem: exc matrix}, we have that $d=5$, $F = E_{\frac{1}{2}}(-5)^{\oplus 2}$ and  $W_{\bullet}$ is the following complex  $$W_{\bullet}: \quad 0\to \OO_{\PP^2}(-7) \to 0,$$ sitting in degree $-1$.

\medskip\noindent
It follows that the log Chern character $\xi(W_{\bullet}) = (1,-7,0)$. Also,
\begin{equation*}
    \begin{aligned}
 n_1=&-\chi(\mathcal{O}_{\PP^2}(8),W^*)=0, &\quad n_2=&\chi(\mathcal{O}_{\PP^2}(-5),F)=6,\\
l_1=&\chi(\mathcal{T}_{\PP^2}(6),W^*)=0, &\quad l_2=& -\chi(\mathcal{T}_{\PP^2}(-6),F)=-2,\\
j_1=&\chi(\mathcal{O}_{\PP^2}(7),W^*)=1, &\quad j_2=&-\chi(\mathcal{O}_{\PP^2}(-4),F)=0.
    \end{aligned}
\end{equation*}

\medskip\noindent
By Proposition \ref{lem: exc matrix}, the minimal free resolution of $U_D$ is therefore
$$0\to \OO_{\PP^2}(-7)\oplus \OO_{\PP^2}(-6)^2\overset{M}{\to} \OO_{\PP^2}(-5)^{6} \to U_D\rightarrow 0,$$ where the matrix $M$ fits the matrix that minimally resolves $E_{\frac{1}{2}}(-5)^{\oplus 2}$.
In particular, $D_{\tiny{Gaeta}^D}$ is a divisor and it spans the secondary edge of $\EFF(M(\xi))$.

\section{Movable cone of the Hilbert scheme $\PP^{2[n]}$}
\label{sec: mov}
\medskip\noindent
In this section, we provide a new computation of the cone of movable divisors $\MOV(\PP^{2[n]})$, when the minimal free resolution of a general $I_Z\in \PP^{2[n]}$ is pure. 
This happens if and only if $n$ is either a triangular number or a tangential number; see Remark \ref{DivGaeta}. The cone $\Mov(\PP^{2[n]})$ was computed in \cite{LZ} for all $n$ via Bridgeland stability conditions. 
Our approach is different as we use minimal free resolutions and the interpolation program described in the introduction.

\medskip\noindent
We recall that for any $n$, the movable cone $\MOV(\PP^{2[n]})$ is generated by two extremal rays. One of them is generated by the class of the family of all subschemes $Z\in \PP^{2[n]}$ with $Z\cap l\ne \emptyset$, where $l$ denotes a fixed line. The other extremal ray of $\MOV(\PP^{2[n]})$ often coincides with the extremal ray of the effective cone $\EFF(\PP^{2[n]})$ and was computed in \cite{CHW}. 
The cases for which the extremal rays of $\MOV(\PP^{2[n]})$ do not match any of those of $\EFF(\PP^{2[n]})$ are the cases in which $D_{\tiny{Gaeta}}\subset \PP^{2[n]}$ is a divisor; see Definition \ref{DDGaeta}. The cases we analyze here (when Gaeta's minimal free resolution of a general $I_Z\in \PP^{2[n]}$ is pure) fall into this category. It remains to work out the other cases using these methods.

\medskip\noindent
Our description of an extremal divisor of the movable cone consists of two steps. First, we find a bundle $E$ that solves the interpolation problem with respect to the general element of $D_{\tiny{Gaeta}}$.
We then proceed to show that the induced Brill-Noether divisor $D_E$ spans an extremal ray of $\MOV(\PP^{2[n]})$. 
The first step is the difficult one, mainly because being numerically orthogonal $\chi(E\otimes I_Z)=0$ does not necessarily imply that $E\otimes I_Z$ has no cohomology. We single out this property following \cite{Hui14}.

\begin{definition}
A vector bundle $E$ on $\PP^2$ of rank $r$ with $h^0(E)=rn$ is said to satisfy interpolation with respect to $I_Z\in \PP^{2[n]}$ if $E\otimes I_Z$ has no cohomology.
\end{definition}

\medskip\noindent
In order to exhibit the solution to the interpolation problem for a general point $I_Z\in D_{\tiny{Gaeta}}\subset \PP^{2[n]}$ when $n$ is either a triangular or a tangential number we proceed as follows. First, we write the ideal sheaf $I_Z\in D_{\tiny{Gaeta}}$ inside a triangle
\begin{equation}\label{orthogonals}
 W \rightarrow F \rightarrow I_Z\rightarrow .
\end{equation}
where $F$ and $W$ are induced by the minimal free resolution of $I_Z$.

Second, we numerically find the character $\zeta$, in the $(\mu,\Delta)$-plane, that satisfies 
$$\chi(\zeta^*, Ch(F))=\chi(\zeta^*,Ch(W))=0,$$ as long as the rank of the log Chern characters $Ch(F)$ and $Ch(W)$ are both bigger than zero, or
$$\chi( \zeta^*,Ch(F))=\chi(\zeta^*, Ch(I_Z))=0,$$  if $Ch(W)$ has rank zero. Third, we prove that a generic bundle $E$ with character $\zeta$ in fact satisfies interpolation with respect to $I_Z\in D_{\tiny{Gaeta}}$.

\medskip\noindent
The first and second steps above are explained in Theorem \ref{TRI} for triangular numbers and in Theorem \ref{MOV} for tangential numbers. Interpolation is proved in Lemma \ref{InterTri} for triangular numbers and in Lemma \ref{InterpolationTan} for tangential numbers. Since a log Chern character $\zeta$ is enough information to compute the minimal free resolution of a generic bundle $E\in M(\zeta)$, we then begin each of the previous lemmas with the minimal free resolution of $E$.

\begin{lem}\label{InterTri}
Let $n=r(r+1)/2$ be a triangular number. 
The general bundle with resolution \[0\rightarrow \mathcal{O}_{\PP^2}(r-3)^{kr}\rightarrow \mathcal{O}_{\PP^2}(r-2)^{k(2r-1)}\rightarrow E\rightarrow 0\]
satisfies interpolation with respect to the general point in $D_{\tiny{Gaeta}}\subset \PP^{2[n]}$ if $k$ is sufficiently large.
\end{lem}

\begin{proof} Let $I_Z\in \PP^{2[n]}$ and observe that $\chi(E\otimes I_Z)=0$. We show the vanishing of the cohomology groups $h^i(E\otimes I_D)$, $i=0,1,2$, when $I_D\in D_{\tiny{Gaeta}}$, by reducing the problem to a computation on an irreducible plane rational curve of degree $r-1$.
Given a curve $C\subset \PP^2$ of degree $r-1$, there is an exact sequence
\[0\to E(-r+1)\to E\to E|_C\to0.\]
From the minimal free resolution of $E$ above, we see that $H^0(\mathbb{P}^2,E(-r+1))=H^1(\mathbb{P}^2,E(-r+1))=0$. 
This implies that the restriction morphism
\[H^0(\mathbb{P}^2,E)\to H^0(C,E|_C)\]
is an isomorphism. Let $C$ be the image of a general map $f:\mathbb{P}^1\to\mathbb{P}^2$ of degree $r-1$. Applying \cite[Theorem 2.8]{Hui13} on $E(-r+1)$, noting that 
$\frac{r}{r-1}\in\Phi_2$, we have that 
\[f^*E\cong\mathcal{O}_{\mathbb{P}^1}(r^2-2r+2)^{k(r-1)}\]
for $k$ sufficiently large.

\medskip\noindent 
The curve $C$ has exactly $\binom{r-2}{2}$ nodes as $f$ is general.  Let $D\subset C$ be the divisor which consists of these nodes in addition to $n-\binom{r-2}{2}=3r-3$ general points of $C$.  Then, we have
\[\deg f^*D=2\binom{r-2}{2}+3r-3=r^2-2r+3.\]
Observe that a section of $E|_C$ vanishing on $D$ induces a section of $f^*E$ vanishing on $f^*D$ and vice versa. In other words,
\begin{align*}
    H^0(C,E|_C(-D))&\cong H^0(\mathbb{P}^1,(f^*E)(-f^*D))\\
    &\cong H^0(\mathbb{P}^1,\mathcal{O}_{\mathbb{P}^1}(-1))^{k(r-1)}=0.
\end{align*}
It follows that $H^0(\mathbb{P}^2,E\otimes\mathcal{I}_D)\cong H^0(C,E|_C(-D))=0$. 

\medskip\noindent
Notice that by tensoring the short exact sequence \[0 \to \mathcal{I}_D \to \mathcal{O}_{\PP^2} \to \mathcal{O}_D \to 0\] with $E$, it follows that $H^2(\mathbb{P}^2,E\otimes\mathcal{I}_D) = 0$. Therefore, $E$ satisfies interpolation for $\mathcal{I}_D$.
Since $D_{\tiny{Gaeta}}$ is irreducible, $E$ then satisfies interpolation for the general element in it.
\end{proof}

\medskip\noindent
Observe that the character of the vector bundle $E$ above is on the curve, in the $(\mu,\Delta)$-plane, of characters numerically orthogonal to the character $ch(I_Z)=(0,n)$. 
However, $E$ is not a minimal-slope stable bundle with this property, which is a crucial property in determining the edge of the effective cone $\EFF(\PP^{2[n]})$.
We now use the bundle $E$ to give $\mathrm{Mov}\left(\mathbb{P}^{2[n]}\right)$ when $n$ is a triangular number.

\begin{theorem}\label{TRI} Let $n=r(r+1)/2$ be a triangular number and $n>3$. The (primary) extremal edge of the movable cone $\MOV(\PP^{2[n]})$ is spanned by the divisor class  \[D_{\mathrm{mov}} = \tfrac{r^2 - 2 r + 2}{r - 1}H-\tfrac{1}{2}B= \left(r-1+\tfrac{1}{r - 1}\right)H-\tfrac{1}{2}B .\]
\end{theorem}
\begin{proof}
Observe that the minimal free resolution of a general ideal sheaf in $\PP^{2[n]}$ is pure: \[0 \to \mathcal{O}_{\PP^2}(-r-1)^r \to \mathcal{O}_{\PP^2}(-r)^{r+1} \to \mathcal{I}_Z \to 0.\] It follows from the Remark \ref{DDGaeta} that the minimal free resolution of a general element $\mathcal{I}_Z$ in $D_{\tiny{Gaeta}}$ is the \textit{divisorial Gaeta} resolution: 
\[0\rightarrow \mathcal{O}_{\PP^2}(-r-2)\oplus \mathcal{O}_{\PP^2}(-r-1)^{r-3}\rightarrow \mathcal{O}_{\PP^2}(-r)^{r-2}\oplus \mathcal{O}_{\PP^2}(-r+1)\rightarrow \mathcal{I}_Z\rightarrow 0.\]

\medskip\noindent
In the derived category $D^b(\PP^2)$, this resolution is equivalent to the following triangle
\[W\rightarrow \mathcal{O}_{\PP^2}(-r+1)\rightarrow \mathcal{I}_Z\rightarrow .\]
where 
$0\rightarrow \mathcal{O}_{\PP^2}(-r-2)\oplus \mathcal{O}_{\PP^2}(-r-1)^{r-3}\rightarrow \mathcal{O}_{\PP^2}(-r)^{r-2}\rightarrow W \rightarrow 0. $ 

Consider the intersection point in the $(\mu,\Delta)$-plane of the following two parabolas 
\begin{equation*}
    \begin{aligned}
 \chi(\zeta\otimes \OOT(1-r))&=0, \\
 \chi(\zeta\otimes I_Z)&=0.
    \end{aligned}
\end{equation*}
These two equations determine the logarithmic Chern character of a bundle $E$ as in Lemma \ref{InterTri}. It follows from this Lemma \ref{InterTri} that the general bundle $E$ with resolution
\[0\rightarrow \mathcal{O}_{\PP^2}(r-3)^{kr}\rightarrow \mathcal{O}_{\PP^2}(r-2)^{k(2r-1)}\rightarrow E\rightarrow 0,\]
satisfies interpolation for a general point in $D_{\tiny{Gaeta}}$. 
The Brill-Noether divisor $D_E$ induced by $E$ has class \[D_E=\tfrac{r^2-2r+2}{r-1}H-\tfrac{1}{2}B,\] and does not contain $D_{\tiny{Gaeta}}$ in its stable base locus and so is movable. 
Furthermore, this class is dual to the curve induced by a general pencil of $n$ points on a curve of degree $r-1$. 
Since this curve sweeps out $D_{\tiny{Gaeta}}$, it follows that the class $D_E$ is extremal in the movable cone $\Mov(\PP^{2[n]})$.
\end{proof}
\medskip\noindent

\begin{coro}
The bundle $E$ of Lemma \ref{InterTri} solves the interpolation problem for the generic $I_Z\in D_{\tiny{Gaeta}}\subset \PP^{2[n]}$, where $n=\tfrac{r(r+1)}{2}$, $r>1$.
\end{coro}

\medskip\noindent
Showing interpolation in the tangential number case is more difficult. In fact, rather than showing interpolation for the general element of $D_{\tiny{Gaeta}}$ as before, we do so for special elements.

\medskip\noindent For $n=2s(s+1)$, according to \cite[Theorem 3.13]{E}, there exist ideals of $n$ points with minimal free resolution
\begin{align}\label{qk}
0\to\OO_{\PP^2}(-2s-2)^s\oplus\OO_{\PP^2}(-2s-1)^k\xrightarrow{\phi}\OO_{\PP^2}(-2s-1)^k\oplus\OO_{\PP^2}(-2s)^{s+1}\to\mathcal{I}_Z\to 0,
\end{align}
as long as $k\leq s$. 

\begin{lem}\label{lem: qk sections}
A scheme $Z$ with resolution \eqref{qk} satisfies $h^0(\PP^2,\mathcal{T}_{\PP^2}(2s-2)\otimes\mathcal{I}_Z)= k$.
\end{lem}
\begin{proof}
Since $\mathcal{T}_{\PP^2}(m)$ has no cohomology for $m=-2,-4$,  we get the following exact sequence after tensoring \eqref{qk} with $\mathcal{T}_{\PP^2}(2s-2)$ and taking cohomology: 
\[0\to H^0(\PP^2,\mathcal{T}_{\PP^2}(2s-2)\otimes\mathcal{I}_Z)\to H^1(\PP^2,\mathcal{T}_{\PP^2}(-1))^k\xrightarrow{\phi_*}H^1(\PP^2,\mathcal{T}_{\PP^2}(-1))^k.\]

The morphism $\phi_*$ is induced by $\phi$ in \eqref{qk} and is the zero map since \eqref{qk} is minimal. 
Therefore, \[h^0(\PP^2,\mathcal{T}_{\PP^2}(2s-2)\otimes\mathcal{I}_Z)=k\cdot h^1(\PP^2,\mathcal{T}_{\PP^2}(-1))=k.\]
\end{proof}

\begin{rmk}
Observe that Lemma \ref{lem: qk sections} writes the minimal free resolution of a set of points that uniquely determine a section of the tangent bundle $\mathcal{T}_{\PP^2}(2s-2)$. In this setting, this answers the Campillo-Olivares problem about the folliations on $\PP^2$ which are fully determined by a subscheme of its singular locus \cite{CO}. We thank  J. Olivares for pointing that out to us.
\end{rmk}

\begin{lem}\label{InterpolationTan}
Let $n=2s(s+1)$ be a tangential integer.
Then the general bundle with the resolution
\[0\rightarrow \mathcal{O}_{\PP^2}(2s-3)^{ks}\rightarrow \mathcal{O}_{\PP^2}(2s-1)^{k(5s-1)}\rightarrow M\rightarrow 0\]
satisfies interpolation with respect to the general point in $D_{\tiny{Gaeta}}$ for $k$ sufficiently large.
\end{lem}

\begin{proof}
Note $D_{\tiny{Gaeta}}$ is an irreducible divisor that contains all sheaves which do not fit into a Gaeta triangle (\ref{res3}) by Corollary \ref{cor: DGaeta}. 
It suffices to provide a single point in $D_{\tiny{Gaeta}}$ whose ideal sheaf satisfies interpolation for the bundle $M$. We do this next.

\medskip\noindent
For $1<s<6$, interpolation can be proven using Macaulay2
and the appendix A contains code in order to do so.
For $s\geq 6$, there are ideal sheaves with the resolution 
\begin{equation}\label{SPECIAL2}
0 \xrightarrow{} \mathcal{O}_{\PP^2}(-2s-2)^s \oplus \mathcal{O}_{\PP^2}(-2s-1)^2 \xrightarrow{\phi} \mathcal{O}_{\PP^2}(-2s-1)^2 \oplus \mathcal{O}_{\PP^2}(-2s)^{s+1} \xrightarrow{} \mathcal{I}_Z \xrightarrow{} 0.    
\end{equation}
These are contained in $D_{\tiny{Gaeta}}$ as they do not have Gaeta's resolution.
Since $s\geq 6$, there exist row and column reductions which reduce the matrix for $\phi$ to the form
\[\phi=\left(\begin{array}{c|c}
0    & A(\phi) \\ \hline
B(\phi) & C(\phi)
\end{array}\right),\]
where $B(\phi)$ is the standard matrix for the resolution of $\mathcal{T}_{\mathbb{P}^2}(-2s-1)^2$.
Since we only need a single sheaf which satisfies interpolation, we may assume that $A(\phi)$ is general and therefore defines a vector bundle $V$ by

\begin{align}\label{eq: esp res}0 \xrightarrow{} V \xrightarrow{} \mathcal{O}_{\PP^2}(-2s-2)^s \xrightarrow{A(\phi)} \mathcal{O}_{\PP^2}(-2s-1)^2 \oplus \mathcal{O}_{\PP^2}(-2s)^{s-5} \xrightarrow{} 0.\end{align}
In other words, the minimal free resolution of $\mathcal{I}_Z$, when considered in the derived category, is equivalent to the triangle 
\[\mathcal{T}_{\mathbb{P}^2}(-2s-1)^{2\bullet} \to \mathcal{I}_Z^{\bullet} \to V[1]^{\bullet}\to \cdot \] whose cohomology yields the short exact sequence of sheaves 
\[0\to V\to \mathcal{T}_{\mathbb{P}^2}(-2s-1)^2 \to \mathcal{I}_Z \to 0.\]

\medskip\noindent
It now suffices to show that $M$ is cohomologically orthogonal to $\mathcal{T}_{\mathbb{P}^2}(-2s-1)^2$ as well as to a general $V$.

\medskip\noindent
Since $\chi(M\otimes \mathcal{T}_{\mathbb{P}^2}(-2s-1)) = 0$, then $\mathcal{T}_{\mathbb{P}^2}(-2s-1)$ satisfies interpolation with respect to a general $M$ as it is an exceptional bundle \cite{CHK}.  
This also can be seen directly from the resolution of $M$ since $\mathcal{T}_{\mathbb{P}^2}(-2)$ and $\mathcal{T}_{\mathbb{P}^2}(-4)$ have no cohomology.

\medskip\noindent
To show interpolation for $V$, we may apply Theorem 1.2 in \cite{CHK}. 
In order to do that, we need to argue that $V$ is a general stable bundle in a positive dimensional moduli space. 
Since $A(\phi)$ is general and is the dual of a Gaeta resolution map, it suffices to show that the discriminant $\Delta(V)$ is at least one as that is above the Dr\'ezet-Le Potier curve (\cite{DL}).
Using the short exact sequence defining $V$ we compute $\Delta(V) = \frac{2s^2-10s+5}{9}$.
For $s\geq 6$, this discriminant is greater than 1 as desired.
Thus, $V$ and $M$ satisfy the hypotheses of Theorem 1.2 in \cite{CHK}.

\medskip\noindent
In order to see which case of Theorem 1.2 in \cite{CHK} we fall into, we consider the primary controlling exceptional bundle to $V$.
Recall that the primary controlling exceptional bundle associated to $V$ is the exceptional bundle for which the solution to the equations \[\chi(V \otimes F)=\chi(V^*, F) = 0 \quad \text{ and }\quad  \Delta(F) = \frac{1}{2}\] lies between its left and right endpoint Chern characters.
Since $\mathrm{rk}(V) = 3$ and $\mu(V)  = \frac{2-8s}{3}$, solving these equations for the slope of $F$ gives
\[\mu(F) = \frac{1}{6} \left(\sqrt{16 s^2 - 80 s + 85} + 16 s - 13\right).\]
If $s$ is congruent to $0 \mod 3$, the left and right endpoints of the interval controlled by $\mathcal{O}_{\PP^2}(\frac{10}{3}s-4)$ are $\frac{1}{6}(-33 + 3 \sqrt{5} + 20 s)$ and $\frac{1}{6} (20 s - 3 (5 + \sqrt{5}))$, respectively.
The slope $\mu(F)$ is between those endpoints if and only if $s \geq 4 = \left\lceil \frac{15 \sqrt{5} - 45}{3 \sqrt{5} - 10}\right\rceil$.
If $s$ is congruent to $1 \mod 3$ ($2 \mod 3$), a similar process shows that the primary controlling exceptional bundle is $E_{\frac{1}{2}}(\frac{10s-13}{3})$ (or $\mathcal{O}_{\PP^2}(\frac{10s-11}{3})$).  
As a result, we have to address each of these cases separately.

\medskip\noindent
First, if $s$ is congruent to $0$ mod $3$, then we have $\chi\left(V\left(\frac{10}{3}s-4\right)\right) = \frac{9 - s}{3}<0$.
By Part 1 of Theorem 1.2 in \cite{CHK}, the cohomology of $M\otimes V$ is determined by $\chi(M\otimes V) = 0$.

\medskip \noindent 
Next, if $s$ is congruent to $1$ mod $3$, then we have $\chi\left(V\otimes E_{\frac{1}{2}}\left(\frac{10s-13}{3}\right)\right) = 2 > 0$.
Since \[\mathrm{rank}(V) - \chi\left(V \otimes E_{\frac{1}{2}}\left(\frac{10s-13}{3}\right)\right) \mathrm{rank}\left(E_{\frac{1}{2}}\left(\frac{10s-13}{3}\right)^*\right)= 3-2\cdot2 = -1<0,\] the cohomology of $M\otimes V$ is determined by $\chi(M\otimes V) = 0$, by Part 2 of Theorem 1.2 in \cite{CHK}.

\medskip\noindent
Finally, if $s$ is congruent to $2$ mod $3$, then we have $\chi\left(V\left(\frac{10s-11}{3}\right)\right) = \frac{s+4}{3}>0$.
Since \[\mathrm{rank}(V) - \chi\left(V\left(\frac{10s-11}{3}\right)\right)\mathrm{rank}\left(\mathcal{O}_{\PP^2}\left(\frac{11-10s}{3}\right)\right)= 3-1\cdot(s+4) = -1-s<0\] 
the cohomology of $M\otimes V$ is determined by $\chi(M\otimes V) = 0$, by Part 2 of Theorem 1.2 in \cite{CHK}.

\medskip\noindent
In every case, the bundle $M$ satisfies interpolation with respect to $V$ and hence for a general $\mathcal{I}_Z$ in $D_{\tiny{Gaeta}}$.
\end{proof}


\medskip\noindent
In order to compute the edge of the movable cone we need the following technical lemma.

\begin{lem}\label{lem: resolution zeroes tangent}
    Let $\gamma\in H^0(\PP^2,\mathcal{T}_{\PP^2}(2s-2))$ be a general section and let $\Gamma$ be the zero locus of $\gamma$. Then the ideal sheaf $\mathcal{I}_{\Gamma}$ has minimal free resolution
    \begin{equation}\label{res zero section}
        0\to\OO_{\PP^2}(-(4s-1))\oplus\OO_{\PP^2}(-(2s+1))\to\OO_{\PP^2}(-2s)^3\to\mathcal{I}_\Gamma\to0.
    \end{equation}
\end{lem}
\begin{proof}
    Observe that if $\Gamma$ is any set of points with the above resolution, then $h^0(\PP^2,\mathcal{I}_\Gamma\otimes \mathcal{T}_{\PP^2}(2s-2))=1$. Therefore, the locus of ideals $I_Z\in \PP^{2[n]}$ with the minimal free resolution above is contained in the locus of subschemes that are the zeroes of a section of $\mathcal{T}_{\PP^2}(2s-2)$. The lemma follows if we show that these two families have the same dimension since they are both irreducible. This is what we do next.

    \medskip\noindent
    Let $\Gamma$ be the zero locus of a general section $\gamma$ and let $Z\subset\Gamma$ be a subscheme of length $n = 2s(s+1)$. Then $\gamma\in H^0(\PP^2,\mathcal{I}_Z\otimes \mathcal{T}_{\PP^2}(2s-2))\neq0$, so $Z\in D_{\tiny{Gaeta}}$. 
    Conversely, given a general element $Z\in D_{\tiny{Gaeta}}$ there is exactly one section (up to scalar multiple) in $H^0(\PP^2,\mathcal{I}_Z\otimes \mathcal{T}_{\PP^2}(2s-2))$ which defines a scheme $\Gamma\supset Z$ as its zero locus.
    This defines a rational, generically finite, and dominant map from $D_{\tiny{Gaeta}}$ to the locus of schemes which are zeroes of a section of $\mathcal{T}_{\PP^2}(2s-2)$, which therefore has dimension equal to $\dim D_{\tiny{Gaeta}}=2n-1$.
    
    \medskip\noindent
    On the other hand, by \cite{BS} as explained in \cite{CM}, the dimension of the schemes with minimal free resolution (\ref{res zero section}) can be computed as \[3\cdot\binom{4s-1-2s+2}{2}+3\cdot\binom{2s+1-2s+2}{2}-2-\binom{4s-1-(2s+1)+2}{2}-9+1=2n-1.\]
\end{proof}


\medskip\noindent
We now compute the primary extremal divisor of the movable cone $\mathrm{Mov}\left(\mathbb{P}^{2[n]}\right)$, when $n$ is a tangential number.
\begin{theorem}\label{MOV} Let $n=2s(s+1)$, with $s\geq2$.
The (primary) extremal ray of the movable cone $\MOV(\PP^{2[n]})$ is spanned by the divisor class  \[D_{mov} = \tfrac{8 s^2 - 4 s + 1}{4 s - 1}H-\tfrac{1}{2}B.\]
\end{theorem}

\begin{proof} Let us write the minimal free resolution of  a general $Z\in D_{\tiny{Gaeta}}$ in the derived category as the following triangle
\begin{equation}\label{QE}
    \begin{aligned}
   W_2\to \mathcal{T}_{\mathbb{P}^2}(-2s-1)\to\mathcal{I}_Z\to   \cdot
    \end{aligned}
\end{equation}
where $W_2$ is considered as a complex of rank $1$,
$$0\to W_2\to\OOT(-2s-2)^s \to \OOT(-2s-1)\oplus\OOT(-2s)^{s-2}\to 0.$$

\medskip\noindent
The log character $\zeta$ that satisfies the following equations
\begin{equation*}
    \begin{aligned}
 \chi(\zeta\otimes \mathcal{T}_{\mathbb{P}^2}(-2s-1))&=0,\\       
 \chi(\zeta\otimes W_2)&=0,
    \end{aligned}
\end{equation*}
is that of a bundle in Lemma \ref{InterpolationTan}. 

\medskip\noindent
By Lemma \ref{InterpolationTan}, the general bundle with resolution \[0\rightarrow \mathcal{O}_{\PP^2}(2s-3)^{ks}\rightarrow \OOT(2s-1)^{k(5s-1)}\rightarrow M \rightarrow 0\] satisfies interpolation with respect to the general point in $D_{\tiny{Gaeta}}$. 
The Brill-Noether divisor induced by the bundle $M$, whose class is $D_{M}=\mu_M H-B/2$, does not contain $D_{\tiny{Gaeta}}$ in its base locus \cite{CHW}, and so is movable.

\medskip\noindent
In order to finish the proof, we need to show that $D_M$ is extremal in $\Mov\left(\mathbb{P}^{2[n]}\right)$. 
We do this next by exhibiting $D_M$ as the dual class of a curve that sweeps out an open set of $D_{\tiny{Gaeta}}$. Consider a general element $Z\in D_{\tiny{Gaeta}}$ as part of the vanishing locus of a section $ H^0(\PP^2, \mathcal{T}_{\PP^2}(2s-2))$, which we denote by $\Gamma=Z\cup\Gamma_{\tiny{res}}$.
Note that $\Gamma_{\tiny{res}}$ has length
\[\delta:=c_2(\mathcal{T}_{\PP^2}(2s-2))-n=2s^2-4s+1.\]
The space of degree $4s-1$ curves which pass through $Z$ and are nodal at $\Gamma_{res}$ has dimension
\[\binom{4s-1+2}{2}-1-n-3\delta=12s-4.\]
Therefore, we may consider an irreducible curve $C$ of degree $4s-1$ which passes through $Z$ and has simple nodes at $\Gamma_{res}$ as its only singularities. Using the Riemann-Roch theorem, we get
\begin{align*}
    h^0(C,\mathcal{O}_C(Z))&=h^0(C,\omega_C \otimes \mathcal{O}_C(-Z))+n+1-g(C)\\
    &=h^0(\PP^2,\mathcal{I}_{\Gamma}(4s-4))+n+1-\binom{4s-2}{2}+\delta\\
    &=h^0(\PP^2,\mathcal{I}_{\Gamma}(4s-4))-(4s^2-8s+1).
\end{align*}
We know that $\Gamma$ has minimal free resolution as stated in Lemma \ref{lem: resolution zeroes tangent} which, in particular, implies that
\[h^0(\PP^2,\mathcal{I}_{\Gamma}(4s-4))=4s^2-8s+3.\]
This proves that $h^0(C,\mathcal{O}_C(Z))=2$ and hence $Z$ moves in a pencil on $C$; which can be proved to be base point free in a similar fashion. This pencil induces a curve $\beta\subset\mathbb{P}^{2[n]}$ which contains $Z$ and satisfies
\begin{align*}
    \beta\cdot H &= 4s-1,\\
    \beta\cdot B/2&=8s^2-4s+1.
\end{align*}
Therefore, $D_{mov}\cdot\beta = 0$, and $D_{mov}$ is extremal in $\MOV(\PP^{2[n]})$.
\end{proof}

\begin{coro}
The bundle $M$ of Lemma \ref{InterpolationTan} solves the interpolation problem for the generic point in $D_{\tiny{Gaeta}}\subset \PP^{2[n]}$, when $n=2s(2+1), s>1$.
\end{coro}

\medskip\noindent
\begin{rmk}\label{eg12} As an example of the distinctive geometry of the families of points in the previous proof, let us describe the dual curve to $D_{mov}\subset \PP^{2[12]}$; denoted above by $\beta$.  If $Z\in \PP^{2[12]}$ is general, then the ideal sheaf $\mathcal{I}_Z$ has Gaeta's minimal free resolution. Moreover, there exist irreducible curves $C\subset \PP^2$ of degree 7 with 1 node as their only singularity that contain $Z$. Consequently, $Z\in C$ can be thought of as an effective divisor of degree 12 on $C$ and as such satisfies that $h^0(\mathcal{O}_C(Z))=1$. However, if $\Gamma \in \PP^{2[12]}$ is a generic point of the extremal divisor $D_{\tiny{Gaeta}}\subset \PP^{2[12]}$, then its ideal sheaf has the following minimal free resolution 
$$0\rightarrow \mathcal{O}_{\PP^2}(-6)^2\oplus \mathcal{O}_{\PP^2}(-5)\overset{M}{\longrightarrow} \mathcal{O}_{\PP^2}(-5)\oplus \mathcal{O}_{\PP^2}(-4)^3\rightarrow \mathcal{I}_{\Gamma}\rightarrow 0,$$
with $M$ a general map.

\medskip\noindent
The $3$ quartics among the generators of the ideal $I_{\Gamma}$ determine 13 points: $\Gamma\cup \{q\}$. In this case, an irreducible curve $C_0$ of degree $7$ with a node as its unique singularity at $q$ satisfies \[h^0(\mathcal{O}_{C_0}(\Gamma))=2.\] 
Therefore, $\Gamma$ can vary in a pencil on $C_0$ which induces a curve $\beta\subset \PP^{2[12]}$ dual to the extremal movable class  $D_{mov}$ in $\MOV(\PP^{2[12]})$. 
\end{rmk}

\medskip\noindent
In this section, in order to exhibit the edge of the movable cone, we had to address the interpolation problem for a Gaeta-admissible object as well as for a general point on the extremal effective divisor. In order to further compute the stable base locus decomposition of $\EFF(\PP^{2[n]})$ one would need to solve the interpolation problem for several loci on the Hilbert scheme. In fact, in the tangential case, the stable base loci of $D_{\tiny{Mov}}$, the extremal movable divisor class contains the configuration of $n-1$ points on a curve of degree $2s-1$. 
Indeed, the curve induced by a general pencil of such configuration intersects the class of $D_{\tiny{Gaeta}}$ negatively in $\mathbb{P}^{2[n]}$. 
Observe that this configuration of points corresponds precisely to sheaves $\mathcal{I}_Z$ which do not admit an exact triangle as in Theorem \ref{MOV}
\begin{equation*}
    \begin{aligned}
   W_2\to \mathcal{T}_{\mathbb{P}^2}(-2s-1)\to\mathcal{I}_Z\to   \cdot
    \end{aligned}
\end{equation*}
Therefore, we may keep running the interpolation program described in the introduction by finding a bundle that satisfies interpolation with respect to sheaves in the locus of $n-1$ points on a curve of degree $2s-1$. We do this in Section \ref{section5} in the case of $\PP^{2[12]}$.


\medskip\noindent
\begin{rmk}
For the cases we did not study, where $n$ is not a triangular or tangential number and which are not covered by \cite{CHW}, we conjecture an exact relationship between the largest component of the stable base locus of the movable cone $\MOV({\PP^{2[n]}})$ and the the minimal free resolutions of ideal sheaves of points in terms of Bridgeland stability conditions. Observe that the relevant cases are those for which the controlling exceptional bundle $E_{\alpha\cdot\beta} = E_{\mu}^*(d)$, where $0\leq \mu <1$, is orthogonal to the general ideal sheaf, i.e. $\chi(E_{\mu}^*(d)\otimes \mathcal{I}_n) =0$.
\begin{conj}
Let $E_{\alpha\cdot \beta}$ be an exceptional bundle of rank at least three.
Let $n\in \mathbb{N}$ be such that \\$\chi\left(E_{\alpha\cdot \beta} \otimes \mathcal{I}_n\right) =0 $ for a general $\mathcal{I}_Z \in M(1,0,n) = \mathbb{P}^{2[n]}$. Let also $E$ denote the exceptional bundle whose slope satisfies $\mu_E\in [0,1)$ and such that $E_{\alpha\cdot \beta} = E^*(d)$. 

\medskip\noindent
If $\mathcal{I}_Z \in M(1,0,n)$ is destabilized at the second last wall by a morphism $F\rightarrow \mathcal{I}_Z$, i.e. it is in the stable base locus of the divisorial chamber but not in the stable base locus of a movable divisor, then it has the Gaeta resolution \[0\rightarrow  \mathcal{O}_{\PP^2}(-d-2)^r \xrightarrow{M}  \mathcal{O}_{\PP^2}(-d-1)^s \oplus \mathcal{O}_{\PP^2}(-d)^t \rightarrow U \rightarrow 0 \text{ if } \mu_E\in \left[0,\frac{1}{2}\right) \text{ or } \]
\[0\rightarrow   \mathcal{O}_{\PP^2}(-d-2)^r \oplus \mathcal{O}_{\PP^2}(-d-1)^s \xrightarrow{M}   \mathcal{O}_{\PP^2}(-d)^t \rightarrow U \rightarrow 0  \text{ if } \mu_E\in \left[\frac{1}{2},1\right),\]
where the matrix $M$ can in both cases be written in the following form 
\[
M=\left(
\begin{array}{c|c}
0& A(M) \\ \hline
 B(M) & C(M)
\end{array}\right)
\]
and the submatrix $B(M)$ yields the map of the minimal free resolution of $F$.
\end{conj}

\medskip\noindent
This conjecture can be verified for any fixed exceptional bundle and the $F$ described in \cite{LZ} which yields the edge of the movable cone. We have done so for all exceptional bundles of rank less than 100.
\end{rmk}

\medskip
\section{Explicit Examples}\label{section5}

\noindent
In the previous two sections, we exhibit how minimal free resolutions of sheaves capture important information of the stable base locus decomposition of $\EFF(M(\xi))$ and also contain information of Bridgeland destabilizing objects. In this section, we further exhibit this phenomenon and explicitly show that Question 1 and Question 2 have affirmative answers for the Hilbert scheme of points $\PP^{2[n]}$, with $n\le 6$ and $n=12$.

\medskip\noindent
In order to exhibit that some Bridgeland destabilizing objects of a sheaf $\mathcal{F}$ are contained in the minimal free resolution of $\mathcal{F}$, we need the following definitions. They are motivated by Bridgeland stability and enhance the notion of $\zeta$-map introduced in Section \ref{sec: eff}.

\begin{definition}
If $M$ is a $\zeta$-map and is not a $\zeta'$-map for any $\zeta'$ with \[\frac{r(U) \mathrm{ch_2}(\zeta)-r(\zeta)\mathrm{ch}_2(U)}{-c_1(\zeta)r(U)}>\frac{r(U) \mathrm{ch_2}(\zeta')-r(\zeta')\mathrm{ch}_2(U)}{-c_1(\zeta')r(U)},\] then we say $M$ is a \textit{general destabilizing} $\zeta$-map.
\end{definition}

\begin{definition}\label{generaldestabilizing}
Let $U$ be a coherent sheaf in $M(\xi)$ with minimal free resolution
\[0\rightarrow \bigoplus^r_{i=1}  \mathcal{O}_{\PP^2}(a_i) \xrightarrow{M} \bigoplus^s_{i=1}  \mathcal{O}_{\PP^2}(b_i)\rightarrow U\rightarrow 0.\] Then $U$ is said to be \textit{stability} $\zeta$\textit{-admissible} if the matrix $M$ is a general destabilizing $\zeta$-map. 
\end{definition}

\medskip\noindent
The definition of $\zeta$-admissible in Section \ref{sec: eff} suffices to study the effective cone $\EFF(M(\xi))$, and in such a case, being stability $\zeta$-admissible is equivalent to $\zeta$-admissible. 

\medskip\noindent
We will use the following notation to organize the information of the tables below:

\begin{itemize}
\item The \textit{stable base locus} of an ideal sheaf $U$ is the smallest stable base locus component that contains $U$.\footnote{ We will also list under this name some loci which are not irreducible components of the stable base locus, but arise by imposing vanishing conditions on the map of their minimal free resolution. To avoid confusion, we will denote such loci as follows: $L\subset \mathbb{P}^{2[n]}$.} \smallskip

\item Let $W$ be the Bridgeland wall at which the sheaf $U$ is destabilized. Then, the set of objects that destabilize $U$ at $W$ will be called the \textit{destabilization set} of $U$.\smallskip

\item The \textit{interpolating bundle} of a sheaf $U$ is the minimal-slope stable bundle which is cohomologically orthogonal to $U$. This is the bundle that solves the interpolation problem for vector bundles with respect to $U$.
\end{itemize}

\medskip\noindent
In the tables below, we emphasize that choosing the matrix representation correctly for a map in a minimal free resolution will allow us to read off properties of the map; similar to a situation of linear algebra. In addition to those, we observe in the tables below, that imposing vanishing conditions on the entries of such matrices allows us to recover the Bridgeland destabilizing objects.

\medskip\noindent
Using the notation above, a strengthening of the results from \cite{LZ} would say that the destabilization set of a sheaf determines its minimal base locus component. We conjecture that the destabilization set of a sheaf comes from setting some of the entries of its minimal free resolution equal to zero.

\begin{conj}\label{conj}
Every element of the destabilization set of a sheaf $U$ comes from vanishing conditions on the map of the minimal free resolution of $U$. In particular, knowing the destabilization set of a sheaf is equivalent to knowing all the characters $\zeta$ such that the sheaf is a stability $\zeta$-admissible sheaf.
\end{conj}

\medskip\noindent
This conjecture would answer Question 1 in the affirmative.
We will now exhibit that this Conjecture holds for the Hilbert schemes of points $\PP^{2[n]}$, with $n\le 6$ and $n=12$. 
In doing so, we will get any possible destabilizing object $i:F\rightarrow I_Z$ that determines a wall in the stable base locus decomposition of $\EFF(\PP^{2[n]})$ from setting entries of the minimal free resolution map equal to zero. Those destabilizing objects are computed following \cite{LZ} and \cite{ABCH}. Further, by taking the kernel of $i$ in the derived category, we can write the  triangle
$$W \rightarrow F \rightarrow \mathcal{I}_Z\rightarrow \cdot $$ 
and see that the minimal free resolution of $\mathcal{I}_Z$ is the mapping cone of two complexes induced by $W$ and $F$; following the proof of Theorem A. 
At this point, we are in the same situation as in the proof of Theorem B. Hence, we may proceed similarly and find the minimal free resolution of a general bundle $E$ that potentially solves the interpolation problem for $\mathcal{I}_Z$. This exemplifies the interpolation program laid out in the introduction and answers Question 2 in the affirmative for the cases of this section.

\medskip\noindent
We carry out these procedures and write the result in the tables below. 
Such tables exhibit the interaction between minimal free resolutions and Bridgeland destabilizing objects as well as the resolutions of the bundles that satisfy interpolation.
The Appendix A contains the Macaulay2 code that can be used to show that all such bundles in fact satisfy interpolation.

\textbf{Notation:}

\begin{itemize}
\item Let $\mathcal{I}_n$ denote the ideal sheaf of $n$ points on $\PP^2$.
    \item Let $L_k(n)$ be the locus of schemes of length $n$ with a linear subscheme of length at least $k$.
    \item Let $Q_k(n)$ be the locus of schemes of length $n$ with a subscheme of length at least $k$ in a conic.
    \item Let $C_k(n)$ be the locus of schemes of length $n$ with a subscheme of length at least $k$ in a cubic.
    \item Let $G_i(n)$ be the locus of subschemes of length $n$ with a fixed Betti diagram. This is a syzygy locus.
\item Primary chambers are those whose boundary rays have non negative coefficient on the divisor $H$ in the basis $H$ and $B$ of the N\'eron-Severi space.
\end{itemize}

    \medskip\noindent
We will use the same notation to represent stable base loci as in \cite{ABCH} in addition to the convention that a horizontal solid line in all the tables below represent a wall in the SBLD; while the horizontal dotted lines separate loci with different sets of minimal free resolutions, or destabilization set, but which are in the same chamber of the SBLD.

\medskip\noindent

\subsection*{Destabilizing objects from minimal free resolutions for $\PP^{2[2]}$}
\medskip\noindent
There is only one Betti diagram for any $\mathcal{I}_Z\in \PP^{2[2]}$:
\begin{align*}G(2)&:0 \to \mathcal{O}_{\PP^2}(-3) \overset{M}{\longrightarrow}\mathcal{O}_{\PP^2}(-2)\oplus\mathcal{O}_{\PP^2}(-1) \to \mathcal{I}_2 \to 0.\end{align*}
Moreover, there is no matrix representation of the map $M$ with a zero entry.

\medskip\noindent
By \cite{ABCH}, the only Bridgeland destabilizing object is $\mathcal{O}_{\PP^2}(-1)$, and it is clear that a sheaf is stability $\zeta$-admissible if and only if it is destabilized by a sheaf with character $\zeta$.

\medskip\noindent
There are no primary chambers with nonempty stable base locus.
Out side the effective cone $\EFF(\PP^{2[2]})$, we say, for convenience, that the base locus is all of $\mathbb{P}^{2[2]}$; which is precisely the locus of all the sub schemes with Betti diagram $G(2)$.
\begin{center}
\begin{tabular}{c|c|c|c} 
Minimal free resolution & Destabilization set & Stable base locus  & Interpolating bundle \\  \hline
$G(2)$ & $\mathcal{O}_{\PP^2}(-1)$ & $\mathbb{P}^{2[2]}$ & -- \\ \hline
\end{tabular}
\end{center}



\medskip
\subsection*{Destabilizing objects from minimal free resolutions for $\PP^{2[3]}$}
\noindent
There are two Betti diagrams we can chose from in order to write the minimal free resolution of any $\mathcal{I}_Z\in \PP^{2[3]}$. They are the following:
\begin{align*}
G_1(3)&:0 \to \mathcal{O}_{\PP^2}(-4)\overset{N}{\longrightarrow}\mathcal{O}_{\PP^2}(-3)\oplus\mathcal{O}_{\PP^2}(-1) \to \mathcal{I}_3 \to 0,\\
G(3)&:0 \to \mathcal{O}_{\PP^2}(-3)^2\overset{M}{\longrightarrow} \mathcal{O}_{\PP^2}(-2)^3 \to \mathcal{I}_3 \to 0.
\end{align*}
There is no matrix representation of the map $N$ which admits an entry equal to zero. However, we can chose a representation of the map $M$, in $G(3)$, which up to row and column reduction generically has zeros on the main diagonal. This situation gives rise to the destabilizing set described below.


\medskip\noindent
By \cite{ABCH}, the Bridgeland destabilizing objects are $\mathcal{O}_{\PP^2}(-1)$, $\mathcal{I}_1(-1)$, $\mathcal{O}_{\PP^2}(-2)$, and $E_{\frac{1}{2}}(-2)$.
Note \cite{ABCH} sometimes excluded cases of higher rank sheaves which yield walls already found; we will not do so.
If an ideal sheaf is destabilized by $E_{\frac{1}{2}}(-2)$, it is also destabilized by $\mathcal{O}_{\PP^2}(-2)$ and $\mathcal{I}_1(-1)$; and vice versa. Ideal sheaves with destabilization set $\mathcal{O}_{\PP^2}(-1)$ and $\{E_{\frac{1}{2}}(-2),\mathcal{I}_1(-1),\mathcal{O}_{\PP^2}(-2)\}$ each have resolution $G_1(3)$ and $G(3)$, respectively.
It is clear that a sheaf is stability $\zeta$-admissible if and only if it is destabilized by a sheaf with character $\zeta$ in all cases.

\medskip\noindent
There is only one primary chamber with nonempty stable base locus. The base locus of that chamber is $L_3(3)=G_1(3)$.
Beyond the primary edge of $\EFF(\PP^{2[3]})$, outside of the effective cone, the base locus is all of $\mathbb{P}^{2[3]}$.

\begin{center}
\begin{tabular}{c|c|c|c} 
Minimal free resolution & Destabilization set & Stable base locus & Interpolating bundle \\  \hline
$G_1(3)$ & $\OOT(-1)$ & $L_3$ & $\mathrm{coker}(\OOT(-2)^4\to\OOT^6)$\\[1.5mm] \hline
$G(3)$ & $\{E_{\frac{1}{2}}(-2),\mathcal{I}_1(-1),\OOT(-2)\}$ & $\mathbb{P}^{2[3]}$ & $\OOT(1)$\\ \hline
\end{tabular}
\end{center}

\bigskip
\subsection*{Destabilizing objects from minimal free resolutions for $\PP^{2[4]}$}
\noindent
There are 3 Betti diagrams we can chose from in order to write the minimal free resolution of any $\mathcal{I}_Z\in \PP^{2[4]}$. They are the following:
\begin{align*}
    G_2(4):0& \to \mathcal{O}_{\PP^2}(-5)\to\mathcal{O}_{\PP^2}(-4)\oplus\mathcal{O}_{\PP^2}(-1) \to \mathcal{I}_4 \to 0,\\
  G_1(4):0 &\to \mathcal{O}_{\PP^2}(-4)\oplus \mathcal{O}_{\PP^2}(-3)\to \mathcal{O}_{\PP^2}(-3)\oplus \mathcal{O}_{\PP^2}(-2)^2 \to \mathcal{I}_4 \to 0,\\
  G(4):0 &\to \mathcal{O}_{\PP^2}(-4)\to \mathcal{O}_{\PP^2}(-2)^2 \to \mathcal{I}_4 \to 0.
\end{align*}
Note that a map of the diagram $G_1(4)$ admits a matrix representation which contains two zero-entries. Indeed, $\OOT(-3)$ to itself and from $\OOT(-4)$ to a copy of $\OOT(-2)$. Except for this case, if $\mathcal{I}_Z$ denotes a generic ideal sheaf which fits into one of the diagrams above, then the map in its minimal free resolution admits no matrix representation with a zero-entry.

\medskip\noindent
By \cite{ABCH}, the Bridgeland destabilizing objects are $\mathcal{O}_{\PP^2}(-1)$, $\mathcal{I}_1(-1)$, $\mathcal{O}_{\PP^2}(-2)$, and $\mathcal{O}_{\PP^2}(-2)^2$. 
An ideal sheaf is destabilized by $\mathcal{O}_{\PP^2}(-2)$ if and only if it is destabilized by $\mathcal{O}_{\PP^2}(-2)^2$.
Ideal sheaves destabilized by $\mathcal{O}_{\PP^2}(-1)$, $\mathcal{I}_1(-1)$, and $\{\mathcal{O}_{\PP^2}(-2),\mathcal{O}_{\PP^2}(-2)^2\}$ each have resolution $G_2(4)$, $G_1(4)$, and $G(4)$, respectively.
It is clear that a sheaf is stability $\zeta$-admissible if and only if it is destabilized by a sheaf with character $\zeta$ in all cases.

\medskip\noindent
The stable base locus decomposition of $\EFF(\PP^{2[4]})$ has two primary cambers with nonempty base locus.
The base locus of the chamber closest to the ample cone $\mathrm{Amp}(\PP^{2[4]})$ is where the four points are collinear; which is precisely the syzygy locus $G_2(4)$.
The base locus of the primary extremal chamber is where at least three of the four points are collinear; this is the union of the syzygy locus $G_2(4)$ and $G_1(4)$.
Beyond the primary edge of $\EFF(\PP^{2[4]})$, the base locus is all of $\mathbb{P}^{2[4]}$. 

\begin{center}

\begin{tabular}{c|c|c|c} 

Minimal free resolution & Destabilization set & Stable base locus & Interpolating bundle\\  \hline
$G_2(4)$ & $\OOT(-1)$ & $L_4$ & $\mathrm{coker}(\OOT(-1)^6\to\OOT^8)$\\[1mm] \hline
$G_1(4)$ & $\mathcal{I}_1(-1)$ & $L_3$ & $\mathrm{coker}(\OOT(-1)^2\to\OOT^2\oplus\OOT(1)^2)$\\[1mm] \hline
$G(4)$ & $\{\OOT(-2),\OOT(-2)^2\}$ & $\mathbb{P}^{2[4]}$ & $\mathrm{coker}(\OOT\to\OOT(1)^3)$\\\hline
\end{tabular}
\end{center}


\bigskip
\subsection*{Destabilizing objects from minimal free resolutions for $\PP^{2[5]}$}
\noindent
There are 3 Betti diagrams we can chose from in order to write the minimal free resolution of any $\mathcal{I}_Z\in \PP^{2[5]}$. They are the following:
\begin{align*}
G_2(5)&:0 \to \mathcal{O}_{\PP^2}(-6)\to\mathcal{O}_{\PP^2}(-5)\oplus\mathcal{O}_{\PP^2}(-1) \to \mathcal{I}_5 \to 0,\\
G_1(5)&:0 \to \mathcal{O}_{\PP^2}(-5)\oplus \mathcal{O}_{\PP^2}(-3)\to \mathcal{O}_{\PP^2}(-4)\oplus \mathcal{O}_{\PP^2}(-2)^2 \to \mathcal{I}_5 \to 0, \\
G(5)&:0 \to \mathcal{O}_{\PP^2}(-4)^2\to \mathcal{O}_{\PP^2}(-3)^2 \oplus \mathcal{O}_{\PP^2}(-2)\to \mathcal{I}_5 \to 0.
\end{align*}
Let us denote map for the Betti diagram $G(5)$ by the matrix $g=(a_{ij})$, where $a_{11},a_{12},a_{21},a_{22}$ are degree 1 polynomials and $a_{31},a_{32}$ are of degree 2. The minimal free resolution of a general sheaf $Z$ admits a map with a zero in one of the degree 2 entries. Further, if $Z$ has a length 3 subscheme contained in a line, we can, for example, impose zeroes in the positions $a_{21},a_{32}$:
$$
g(5)'=\left(
\begin{array}{cc}
*&* \\ 
 0&*\\
 * &0
\end{array}\right).$$

\noindent
We will denote by $\{G(5), g(5)'\}$ the locus of ideal sheaves whose minimal free resolution has Betti diagram $G(5)$ and map $g(5)'$. If $Z$ contains 2 different length 3 subschemes, each contained in a line, its defining matrix admits an alternative representation with zeroes on $a_{12}$ and $a_{21}$:
$$ 
g(5)''=\left(
\begin{array}{cc}
*&0 \\ 
 0&*\\
 *&*
\end{array}\right).
$$
\noindent
Denote by $\{G(5),g(5)''\}$ the locus of such ideal sheaves. The general element with resolution $G_1(5)$ represents a scheme with a degree 4 linear subscheme. It admits the zero in position $a_{12}$, together with a zero in the position $a_{21}$. There are no further vanishings and the same is true for $G_2(5)$.

\medskip\noindent
By \cite{ABCH}, the Bridgeland destabilizing objects are $\mathcal{O}_{\PP^2}(-1)$, $\mathcal{I}_1(-1)$, $\mathcal{O}_{\PP^2}(-2)$, and $\mathcal{I}_2(-1)$.
Note, if an ideal sheaf is destabilized by $\mathcal{I}_2(-1)$, it is also destabilized by $\mathcal{O}_{\PP^2}(-2)$.
Objects destabilized by $\mathcal{O}_{\PP^2}(-1)$, $\mathcal{I}_1(-1)$, $\{\mathcal{I}_2(-1),\mathcal{O}_{\PP^2}(-2)\}$, and $\mathcal{O}_{\PP^2}(-2)$ each have resolution $G_2(5)$, $G_1(5)$, $\{G(5)'', G(5)'\}$, and $G(5)$, respectively.
Finally, $G(5)''$ corresponds to the jumping of the $\hom(\mathcal{I}_2(-1),\mathcal{I}_5)$ from 1 to 2 so it does not add any destabilizing objects.
However, this does show up geometrically as the five points having two sets of three collinear points.
It is clear that if a sheaf is stability $\zeta$-admissible if and only if it is destabilized by a sheaf with character $\zeta$ in all cases.

\medskip\noindent
The stable base locus decomposition of $\EFF(\PP^{2[5]})$ has two primary chambers with nonempty base locus.
The base locus of the chamber closest to the ample cone $\mathrm{Amp}(\PP^{2[5]})$ is where the five points are collinear, which is precisely the syzygy locus $G_2(5)$.
The base locus of the other primary chamber is where at least four of the five points are collinear; this is precisely the union of the syzygy locus $G_1(5)$ and $G_2(5)$. Beyond the primary edge of $\EFF(\PP^{2[5]})$, the base locus is all of $\mathbb{P}^{2[5]}$. 

\medskip\noindent
Note the sheaves in the locus $\{G(5),g(5)'\}$ have the geometric property that three of the five points are collinear, while the sheaves in $\{G(5),g(5)''\}$ have the property that there are two sets of three points which are collinear, i.e. the five points lie on the union of two lines and include the intersection point of the line as one of the five points.
We will use $L_{3,3}$ to denote this final locus.

\begin{center}
\begin{tabular}{c|c|c|c} 
Minimal free resolution & Destabilization set & Stable base locus & Interpolating bundle\\  \hline
$G_2(5)$ & $\OOT(-1)$ & $L_5$ & $\mathrm{coker}(\OOT(-1)^8\to\OOT^{10})$\\[1.5mm] \hline
$G_1(5)$ & $\mathcal{I}_1(-1)$ & $L_4$ & $\mathrm{coker}(\OOT(-1)^4\to\OOT^4\oplus\OOT(1)^2)$\\[1.5mm] \hline
$G(5),g(5)''$ & $\mathcal{I}_2(-1),\OOT(-2)$  & $L_{3,3}\subset \mathbb{P}^{2[5]}$ & $\mathrm{coker}(\OOT^2\to\OOT(1)^4)$ \\ \hdashline
$G(5),g(5)'$ & $\mathcal{I}_2(-1),\OOT(-2)$ & $L_3\subset \mathbb{P}^{2[5]}$ & $\mathrm{coker}(\OOT^2\to\OOT(1)^4)$\\ \hdashline
$G(5)$ & $\OOT(-2)$ & $\mathbb{P}^{2[5]}$ & $\mathrm{coker}(\OOT^2\to\OOT(1)^4)$\\ \hline
\end{tabular}
\end{center}


\bigskip
\subsection*{Destabilizing objects from minimal free resolutions for $\PP^{2[6]}$}
\noindent
There are 5 Betti diagrams we can chose from in order to write the minimal free resolution of any $\mathcal{I}_Z\in \PP^{2[6]}$. They are the following:
\begin{align*}
G_4(6)&:0 \to \OOT(-7)\to\OOT(-6)\oplus\OOT(-1) \to \mathcal{I}_6 \to 0,\\
G_3(6)&:0 \to \OOT(-6)\oplus\OOT(-3)\to \OOT(-5)\oplus \OOT(-2)^2 \to \mathcal{I}_6 \to 0,\\
G_2(6)&:0 \to \OOT(-5)\oplus \OOT(-4)\to\OOT(-4)\oplus\OOT(-3)\oplus \OOT(-2) \to \mathcal{I}_6 \to 0,\\
G_1(6)&:0 \to \OOT(-5)\to \OOT(-3)\oplus \OOT(-2) \to \mathcal{I}_6 \to 0,\\
G(6)&:0 \to \OOT(-4)^3\to \OOT(-3)^4 \to \mathcal{I}_6 \to 0.
\end{align*}

\medskip\noindent
The number of distinct maps grows quickly, then we will only list the maps that give rise to the destabilization set and stable base locus components.  The relevant maps that come from the Betti diagram $G(6)$ arise by writing a 5th zero anywhere in the matrix in addition to the 4 zeros, one in each row, one obtains by row-column reduction on the generic matrix. Let us denote such a map by $g(6)'$. For example:

$$
g(6)'=\left(
\begin{array}{ccc}
0&*&* \\ 
 *&0&*\\
 *&*&0\\
 0&*&0 \\ 
\end{array}\right).
$$

\medskip\noindent
By \cite{ABCH}, the  Bridgeland destabilizing objects are $\OOT(-1)$,  $\mathcal{I}_1(-1)$,   $\mathcal{I}_2(-1)$, $\OOT$,  $\mathcal{I}_1(-2)$, $\mathcal{I}_3(-1)$, $\OOT(-3)$, $E_{\frac{1}{2}}(-3)$, $F_{(2,-2,1)}$, $\OOT(-3)^2$, $\OOT(-3)^3$, and $\OOT(-3)^4$. 
Here $F_{(2,-2,1)}$ denotes a stable bundle with rank $2$, slope $-2$, and discriminant $1$.
Note, if an ideal sheaf is destabilized by any of $E_{\frac{1}{2}}(-3)$, $F_{(2,-2,1)}$, $\OOT(-3)$, $\OOT(-3)^2$, $\OOT(-3)^3$, $\OOT(-3)^4$, $\mathcal{I}_1(-2)$, it is destabilized by all of those. 
Those bundles also destabilize any bundle destabilized by $\mathcal{I}_1(-2)$.
Sheaves destabilized by only all of those bundles but not $\mathcal{I}_1(-2)$ have resolution $G(6)$, while those destabilized by that set and $\mathcal{I}_1(-2)$ have resolution $\{G(6),g(6)'\}$.
Similarly, objects destabilized by $\mathcal{O}_{\PP^2}(-1)$, $\mathcal{I}_1(-1)$,$\mathcal{I}_2(-1)$ and $\mathcal{O}_{\PP^2}(-2)$ each have resolution $G_4(5)$, $G_3(5)$, $G_2(5)$, and $G_1(5)$, respectively.
It is clear that if a sheaf is stability $\zeta$-admissible if and only if it is destabilized by a sheaf with character $\zeta$ in all cases.

\medskip\noindent
The stable base locus decomposition of $\EFF(\PP^{2[6]})$ has four primary chambers with nonempty base locus.
The base locus of these chambers from closest to the ample cone $\mathrm{Amp}(\PP^{2[6]})$ are where the six points are collinear, where the five of the six points are collinear, where the four of the six points are collinear, and where the six points lie on a conic, which is precisely the syzygy loci $G_4(5)$, $G_3(5)$, $G_2(5)$, and $G_1(5)$, respectively.
Beyond the primary edge of $\EFF(\PP^{2[5]})$, the base locus is all of $\mathbb{P}^{2[6]}$.  Note the sheaves in the locus $\{G(6), g(6)'\}$ have the geometric property that three of the six points are collinear.

\medskip\noindent
\begin{center}
\begin{tabular}{c|c|c|c} 
Min free resolution & Destabilization set & SBL & Interpolating bundle\\  \hline
$G_4(6)$ & $\OOT(-1)$ & $L_6$ & $\mathrm{coker}(\OOT(-1)^{10}\to\OOT^{12})$ \\[1.5mm] \hline
$G_3(6)$ & $\mathcal{I}_1(-1)$ & $L_5$ & $\mathrm{coker}(\OOT(-1)^6\to\OO^6\oplus\OOT(1)^2)$\\[1.5mm] \hline
$G_2(6)$ 
& $\mathcal{I}_2(-1)$ &  $L_4$ & $\mathrm{coker}(\OOT(-1)^2\to\OOT(1)^4)$\\[1.5mm] \hline
$G_1(6)$ & $\OOT(-2)$ & $Q_6$ & $\mathrm{coker}(\OOT^3\to\OOT(1)^5)$\\[1.5mm] \hline
$G(6),g(6)'$ 
& $\{E_{\frac{1}{2}}(-3),F_{(2,-2,1)},\OOT(-3),\OOT(-3)^2,$ &  $L_3 \subset \mathbb{P}^{2[6]}$ & $\OOT(2)$\\
& $\OOT(-3)^3,\OOT(-3)^4,\mathcal{I}_1(-2),\mathcal{I}_3(-1)\}$ & \\ \hdashline
$G(6)$ 
& $\{E_{\frac{1}{2}}(-3),F_{(2,-2,1)},\OOT(-3),\OOT(-3)^2,$ & $\mathbb{P}^{2[6]}$ & $\OOT(2)$\\
& $\OOT(-3)^3,\OOT(-3)^4,\mathcal{I}_1(-2)\}$ & \\ \hline
%
\end{tabular}
\end{center}

\medskip
\subsection*{Destabilizing objects from minimal free resolutions for $\PP^{2[12]}$}
\noindent
This example exhibits that the minimal free resolutions recover the destabilization set also when the stable base locus decomposition has reducible base loci. 
We will see that such components correspond to distinct destabilization set.
We will omit the analysis for chambers whose base loci consists of subschemes with $k$ collinear points for $k \ge 6$; those follow from \cite{ABCH} and they have minimal free resolutions which can be related to the destabilization set in a clear way.

\medskip\noindent
There are 10 relevant Betti diagrams for the minimal free resolution of points $\mathcal{I}_Z\in \PP^{2[12]}$ which are the following. Other diagrams correspond to either chambers we have omitted or to special loci inside the irreducible components of the stable base locus decomposition.
\begin{align*}
G_{9}(12): 0 & \to\OOT(-8)\oplus\OOT(-5)^2\to\OOT(-7)\oplus\OOT(-4)^2\oplus\OOT(-3)\to\mathcal{I}_{12}\to0,\\
G_{8}(12): 0& \to \OOT(-8) \to \OOT(-6) \oplus \OOT(-2) \to \mathcal{I}_{12} \to 0,\\
G_{7}(12): 0& \to \OOT(-7)^2 \oplus \OOT(-4) \to  \OOT(-6)^2 \oplus \OOT(-3)^2 \to \mathcal{I}_{12} \to 0,\\
G_{6}(12): 0& \to \OOT(-7)\oplus \OOT(-5)^3 \to \OOT(-6) \oplus \OOT(-4)^4 \to \mathcal{I}_{12} \to 0,\\
G_{5}(12): 0& \to \OOT(-7) \oplus \OOT(-5)\to \OOT(-5) \oplus \OOT(-4)\oplus \OOT(-3) \to \mathcal{I}_{12} \to 0,\\
G_{4}(12): 0& \to \OOT(-6)^3 \to  \OOT(-5)^3 \oplus \OOT(-3) \to \mathcal{I}_{12} \to 0,\\
G_{3}(12): 0& \to \OOT(-6)^2\oplus \OOT(-5)^2 \to \OOT(-5)^2 \oplus \OOT(-4)^3 \to \mathcal{I}_{12} \to 0,\\
G_{2}(12): 0& \to \OOT(-8)\oplus \OOT(-5)^2 \to \OOT(-7) \oplus \OOT(-4)^2\oplus \OOT(-3)  \to \mathcal{I}_{12} \to 0,\\
G_{1}(12): 0& \to \OOT(-6)^2\oplus \OOT(-5) \to \OOT(-5) \oplus \OOT(-4)^3 \to \mathcal{I}_{12} \to 0,\\
G(12): 0& \to \OOT(-6)^2\to \OOT(-4)^3 \to \mathcal{I}_{12} \to 0.
\end{align*}
\noindent 
A relevant map occurs in the diagram $G_1(12)$ and is the map whose entry from $\OOT(-5)$ to one copy of $\OOT(-4)$ is zero. We denote such a map by $g_1(12)'$. Another relevant map occurs in the diagram $G(12)$ is defined by making one of the entries is zero in the general matrix; we denote it by $g(12)'$

\medskip\noindent
There are destabilizing objects denoted by $F_{(2,-3,\frac{3}{2})}$ and $F_{(2,-3,1)}$ which are general stable bundles with their log Chern character as a sub index.

\medskip
\begin{center}
\begin{tabular}{c|c|c|c} 
Min free resolution & Destabilization set & SBL & Interpolating bundle\\  \hline
$G_9(12)$ & $\mathcal{I}_5(-1)$ & $L_7$ & $\mathrm{coker}(\OOT(1)^2\oplus\OOT^6\to\OOT(1)^{10})$\\[1.5mm] \hline 
$G_8(12)$ & $\OOT(-2)$ & $Q_{12}$ & $\mathrm{coker}(\OOT^9\to\OOT(1)^{11})$\\[1.5mm] \hline
$G_{7}(12)$ & $\mathcal{I}_1(-2)$ & $Q_{11}$ & $\mathrm{coker}(\OOT^6\to\OOT(1)^6\oplus\OOT(2)^2)$\\ \hdashline
$G_{6}(12)$ & $\mathcal{I}_6(-1)$ & $L_{6}$ & $\mathrm{coker}(\OOT^6\to\OOT(1)^6\oplus\OOT(2)^2)$\\[1.5mm] \hline
$G_{5}(12)$ & $\mathcal{I}_2(-2)$ & $Q_{10}$ & $\mathrm{coker}(\OOT^3\to\OOT(1)\oplus\OOT(2)^4)$\\[1.5mm] \hline
$G_{4}(12)$ & $\OOT(-3)$ & $C_{12}$ & $\mathrm{coker}(\OOT(1)^4\to\OOT(2)^6)$\\ \hdashline
$G_{3}(12)$ & $\mathcal{I}_3(-2)$ & $Q_{9}$ & $\mathrm{coker}(\OOT(1)^4\to\OOT(2)^6)$\\ \hdashline
$G_{2}(12)$ & $\mathcal{I}_5(-1)$ & $L_{5}$ & $\mathrm{coker}(\OOT(1)^4\to\OOT(2)^6)$\\[1.5mm] \hline
$G_{1}(12), g_1(12)'$ & $\mathcal{I}_1(-3)$ & $C_{11}$ & $\mathrm{coker}(\OOT(1)^2\to\OOT(2)^2\oplus\OOT(3)^3)$\\[1.5mm] \hline
$G_{1}(12)$ & $\mathcal{T}_{\mathbb{P}^2}(-5)$ & $D_{\mathcal{T}_{\mathbb{P}^2}(-5)}$ & $\mathrm{coker}(\OOT(1)^2\to\OOT(3)^9)$\\[1.5mm] \hline
$G(12), g(12)'$ & $\{\OOT(-4),\OOT(-4)^2,\OOT{O}(-4)^3$ &  & $\mathrm{coker}(\OOT(2)\to\OOT(3)^3)$\\[1.5mm]
 & $F_{(2,-3,\frac{3}{2})},\mathcal{I}_4(-2)\}$ &$Q_8 \subset \mathbb{P}^{2[12]}$ & \\[1.5mm]
\hdashline
$G(12)$ & $\{\OOT(-4),\OOT(-4)^2,\OOT(-4)^3,$ &  & $\mathrm{coker}(\OOT(2)\to\OOT(3)^3)$\\
 & $F_{(2,-3,\frac{3}{2})}\}$ & $\mathbb{P}^{2[12]}$ & $\mathrm{coker}(\OOT(2)\to\OOT(3)^3)$\\[1.5mm] \hline
\end{tabular}
\end{center}

\medskip\noindent
In each chamber, solid lines denote chambers while dotted lines do not, we see that the minimal free resolution contains the destabilization set; which itself determines the stable base locus decomposition, including when there are many components within a chamber. 
Similarly, a sheaf is stability $\zeta$-admissible if and only if it is destabilized by a sheaf with character $\zeta$ in all cases.

\section*{Appendix A: Computations in Macaulay2}

\noindent 
This appendix contains an implementation, using Macaulay2, for proving Lemma \ref{InterpolationTan} for the cases $s\leq5$. The cases $s\leq 4$ can be treated uniformly so we only address $s=4$. In the case $s=5$, an optimization is possible using the resolution (\ref{eq: esp res}) in Lemma \ref{InterpolationTan}. Note, this code can also be used to prove interpolation for a vector bundle with respect to a sheaf more generally.

\subsection*{Case $s=4$.}
\begin{verbatim}
PP2 = QQ[x, y, z];
--First we define the bundle M
B2 = flatten entries basis(2, PP2);
C4 = for i in 0..18 list (random(PP2^{2}, PP2^{0}))_0_0;
A = matrix for i in 0..18 list
	for j in 0..3 list(
		if j == 3 then C4_i
		else if j == floor(i/6) then B2_(i-6*j)
		else 0
	);
M = sheaf coker map(PP2^{7}^19, PP2^{5}^4, A);
--Next we define a general ideal in J
B = random(PP2^{-9,5:-8}, PP2^{4:-10,-9});
--This line is necessary to ensure the map is minimal
B = B - sub(B, {x=>0, y=>0, z=>0});

I = fittingIdeal(1, coker B);
--The result should be 0
HH^0(M ** sheaf module I)
\end{verbatim}

\medskip\noindent
    Running the code above in the Institute of Mathematics' computer cluster at UNAM took 62830 seconds. The cases $s<4$ are considerably faster.

\subsection*{Case $s=5$.}
In this case, instead of computing the cohomology groups of $M\otimes\mathcal{I}_Z$, we compute the cohomology of $M\otimes V$, where $V$ is the bundle (\ref{eq: esp res}) in Lemma \ref{InterpolationTan}.
\begin{verbatim}
    PP2 = QQ[x, y, z]

B2 = flatten entries basis(2, PP2);
C4 = for i in 0..23 list (random(PP2^{2}, PP2^{0}))_0_0;
A = matrix for i in 0..23 list
	for j in 0..4 list(
		if j == 4 then C4_i
		else if j == floor(i/6) then B2_(i-6*j)
		else 0
	);
M = sheaf coker map(PP2^{9}^24, PP2^{7}^5, A);

V = sheaf ker random(PP2^{-11}^2, PP2^{-12}^5);

time HH^1(M ** V)
\end{verbatim}

\end{document}